\documentclass{article}

\usepackage[a4paper, total={6in, 8in}]{geometry}
\usepackage{emptypage}
\usepackage[utf8]{inputenc}
\usepackage[T1]{fontenc}
\usepackage{lmodern}
\usepackage[english]{babel}
\usepackage{amsmath, amssymb}
\usepackage{amsthm}
\usepackage{longtable}
\usepackage{placeins}
\usepackage[font=small,labelfont=bf]{caption}
\usepackage{comment}

\usepackage{enumerate}
\usepackage{enumitem}
\usepackage{xcolor}
\RequirePackage[colorlinks,allcolors=blue]{hyperref}
\RequirePackage{graphicx}
\usepackage[nottoc]{tocbibind}

\usepackage{titlesec}
\titleformat{\section}{\large\bfseries}{\thesection}{1em}{}

\usepackage[round]{natbib} 

\newtheorem{theorem}{Theorem}
\newtheorem{lemma}{Lemma}
\newtheorem{corollary}{Corollary}

\newtheorem{example}{Example}

\theoremstyle{definition}

\newtheorem*{remark*}{Remark}

\numberwithin{equation}{section}

\newcommand{\eps}{\varepsilon}

\newcommand{\E}{\mathbb{E}}
\newcommand{\N}{\mathbb{N}}
\newcommand{\R}{\mathbb{R}}
\newcommand{\B}{\mathbb{B}}

\newcommand{\G}{\mathbb{G}}
\renewcommand{\P}{\mathbb{P}}

\newcommand{\Rc}{\mathcal{R}}
\newcommand{\Uc}{\mathcal{U}}
\newcommand{\Dc}{\mathcal{D}}
\newcommand{\Oc}{\mathcal{O}}
\newcommand{\weak}{\rightsquigarrow}

\newcommand{\ir}[1]{#1^\uparrow}		

\title{Limiting properties of monotone rearrangements of estimators when the truth is flat}
\date{\today}
\author{Holger Dette\footnote{Faculty of Mathematics, Ruhr-Universität Bochum, \href{mailto:holger.dette@rub.de}{holger.dette@rub.de}} \and Marius Kroll\footnote{Faculty of Mathematics, Ruhr-Universität Bochum, \href{mailto:marius.kroll@rub.de}{marius.kroll@rub.de}} \and Stanislav Volgushev\footnote{Department of Statistical Sciences, University of Toronto, \href{mailto:stanislav.volgushev@utoronto.ca}{stanislav.volgushev@utoronto.ca}}}
\begin{document}
\maketitle

\begin{abstract}
  Monotone rearrangements provide a simple way to enforce shape constraints of an estimator, but existing distributional theory does not cover flat regions, where the target induces no local ordering. We study rearranged estimators in two canonical flat settings. First, for a histogram estimator of the uniform density, we establish functional weak convergence of its non-decreasing rearrangement at the parametric rate on compact subsets of the interval $(0,1)$ after an additional deterministic centering. This result is strikingly different from what is known for strictly monotone densities. Second, we consider two rearranged estimators of rearranged copulas under independence, based on empirical-copula increments and on a checkerboard approximation. After appropriate centering and rescaling, both estimators converge weakly on $[0,1]^2$ to an integrated Gaussian process which was not known before. We further use these results to  prove asymptotic normality for a broad class of rearranged copula-based dependence measures, which were recently discussed in \cite{Strothmann.2022}.
\end{abstract}

\smallskip

\section{Introduction}
\label{sec1}

Shape restrictions such as monotonicity arise naturally in many statistical problems, including the estimation of densities, regression functions, and conditional quantiles. A simple way of enforcing monotonicity of an estimator is its monotone rearrangement, which  originates in classical majorization theory; see, for example, \cite{HardyLittlewoodPolya1952}. The  appeal of monotone rearrangements in statistics rests in particular on a contraction property: if the target is monotone, rearranging an estimator cannot increase its distance from the target for a broad class of loss functions including $L^p$-metrics \citep{Lorentz1953}. 
Rearrangement methods have been applied to monotone density and regression estimation \citep{Fougeres1997,detneupil2006,ChernozhukovFernandezValGalichon2009,CollingVanKeilegom2019,Wied2024}, to the estimation of the effective dose \citep{DetteNeumeyerPilz2005}, to the construction of non-crossing  quantile curves \citep{DetteVolgushev2008,ChernozhukovFernandezValGalichon2010}, to the anaylsis of conditional quantile processes \citep{QuYoon2015,BelloniChernozhukovChetverikovFernandezVal2019}, to change point detection  \citep{dettewu2019}  and to simultaneous inference in time series
\citep{Luo20102025} among others. A general asymptotic framework for monotone rearrangements of density and regression estimators was developed by \cite{AnevskiFougeres2019}. They established  stability and consistency properties of the rearrangement operator and derived pointwise limiting distributions for rearranged estimators under fairly general stochastic assumptions. The distributional results in this work, however, do not cover points belonging to flat regions of the density or regression function.

The asymptotic behavior of shape-constrained estimators on flat regions has been investigated in several related settings. In the discrete case, \cite{JankowskiWellner2009} studied the monotone rearrangement of the empirical estimator for a decreasing probability mass function. They showed, in particular, that on a constant block the limiting distribution of the rearranged estimator is obtained by rearranging the corresponding Gaussian limit vector. Estimation when the locations of flat regions are known was considered by \cite{AnevskiPastukhov2016}.  Flat regions have also been studied in monotone density estimation for the  Grenander estimator. For example, \cite{GroeneboomPyke1983} investigated the asymptotic distribution of the $L^2$-distance between this estimator and the uniform density, while \cite{CarolanDykstra1999} studied the pointwise asymptotic distribution of the Grenander estimator at points where the underlying density is locally constant. These estimators, however, are based on concave-majorant or projection operations and are therefore fundamentally different from estimators obtained by monotone rearrangements.

Although monotone rearrangements are now a classical tool in statistics, no distributional results on rearranged estimators when the estimation target has flat pieces are available to date. In this paper we deepen our understanding of rearrangement based estimators by investigating the distributional properties of rearranged estimators of a flat target for two important estimation problems. We first consider the non-decreasing  rearrangement of a histogram estimator for the uniform density on the interval $[0,1]$.  Since the ``true''  density is constant, the rearrangement is essentially determined by the order statistics of the stochastic fluctuations of the histogram estimator. This  leads to an additional centering term which cannot be neglected after rescaling.   If  the rearranged estimator is centered by subtracting this term, it converges as a process at the parametric $n^{-1/2}$ rate on every compact interval  of $(0,1)$, and this rate is independent of the number of histogram cells (subject to  growth conditions). We also show that the process convergence cannot be extended to the boundary and derive a Gumbel limit for the supremum of the rearranged estimator.

Our second application concerns estimates of rearranged copulas of copula-based dependence measures, which were recently introduced by \cite{Strothmann.2022}. Due to the complex nature of that estimator, no distributional limit theory is available for it to date.
 For a copula $C$ of a two-dimensional vector $(X,Y)$, these authors constructed a non-decreasing rearrangement of the function $u\mapsto \partial_1 C(u,v)$ for each fixed $v$ and defined the rearranged copula $C^\uparrow$ by integrating the resulting monotone function with respect to $u$. Given a copula-based dependence measure $\mu$, the associated rearranged measure is then defined as $R_\mu(X,Y)=\R_\mu(C):=\mu(C^\uparrow)$. \cite{Strothmann.2022} showed that, for many dependence measures $\mu$ the rearranged dependence  has very desirable properties, in particular
\begin{itemize}
    \item[(1)]  $0\leq R_\mu(X,Y)\leq 1$
 \item[(2)]  $R_\mu(X,Y)=0$ if and only if 
$
X\ \text{and}\ Y\ \text{are independent}$
 \item[(3)] 
$R_\mu(X,Y)=1$ if and only if  $
Y=f(X)\ \text{almost surely for some measurable function }$
\end{itemize}
This construction allows, for example, the construction of a rearranged Kendall's $\tau$ or Spearman's $\rho$ with the properties (1) - (3). Under independence, we have  $ C(u,v)=uv$, and $\partial_1C(u,v)=v,$
and hence, for every fixed $v$, the function $u\mapsto\partial_1C(u,v)$ is constant. Thus, independence corresponds precisely to a flat target in the coordinate with respect to which the rearrangement is performed. 

We investigate the distributional properties of two estimators of $C^\uparrow$ under independence: one based on increments of the empirical copula and one based on the partial derivative of a checkerboard approximation.
 For each fixed value of the second coordinate, these increments are rearranged and subsequently integrated. The second estimator applies the same construction to the partial derivative of a checkerboard approximation of the empirical copula. Under independence, we establish a functional central limit theorem for both  rearranged copula estimators. More specifically,  both estimators converges weakly at the parametric rate to an integrated centered Gaussian process in  $\ell^\infty([0,1]^2)$ (after appropriate centering and rescaling).  Finally, we use a functional delta method to derive asymptotic distributions for the resulting estimators of the  rearranged dependence measures. In particular, asymptotic normality follows whenever the relevant derivative of the copula functional is linear and continuous.

The remainder of the paper is organized as follows. In Section~\ref{sec2}, we study the rearranged histogram estimator of the uniform density and provide a proof of the main result. Section~\ref{sec3} introduces the two estimators of the rearranged copula and derives their limiting distributions under independence, together with the resulting asymptotic theory for rearranged dependence measures. The main key proof steps are presented in Section~\ref{sec:proof_outline}. A number of intermediate technical results that appear in those proofs are proved in Appendix~\ref{sec:proof-cop}.

\section{Monotone rearrangements}
\label{sec:monotone_rearrangements}
Let us briefly recall the definition and some relevant properties of the monotone rearrangement of a real-valued function. In our presentation we follow \citep[Chapter 6]{leoni:2009}, but slightly extend their definition to possibly negative-valued functions. We define for a Borel measurable function $f : [0,1] \to \mathbb{R}$ its decreasing (or non-increasing) rearrangement $\mathcal{R}f : [0,1] \to \bar{\mathbb{R}} = \mathbb{R} \cup \{\pm \infty\}$ by 
\begin{equation}
    \label{eq:decreasing_rearrangement_definition}
\mathcal{R}f(t) = \inf\left\{u \in \mathbb{R} ~|~ \lambda(f > u) \leq t \right\},
\end{equation}
where we define $\inf \emptyset = + \infty$ and $\lambda$ denotes the Lebesgue measure on $[0,1]$. It is the generalised inverse of the decreasing tail probability function $u \mapsto \lambda(f > u)$ \citep[cf.\@ Theorem 1.7 and Remark 1.10 in][]{leoni:2009}. $\mathcal{R}f$ is right-continuous and decreasing (hence Borel measurable). It may take the value $+\infty$ only at $t = 0$ and $-\infty$ only at $t = 1$. If $f$ is bounded, then $\mathcal{R}f$ is bounded, too. For classical results about the decreasing rearrangement such as the Hardy-Littlewood inequality, we refer to \cite{leoni:2009}, Chapter 6. One property which we will use repeatedly is that the rearrangement operator restricted to the space of bounded functions is a contraction with respect to the supremum norm; i.e.\@ $\|\mathcal{R}f - \mathcal{R}g\|_\infty \leq \|f-g\|_\infty$ for any two bounded Borel measurable functions $f,g : [0,1] \to \mathbb{R}$. This is essentially Theorem 1 in \cite{AnevskiFougeres2019}, where it is stated for non-negative bounded functions $f$ and $g$. Our slightly more general statement follows by noting that if $f,g : [0,1] \to [a,b]$ are bounded and possibly negative-valued, then $\mathcal{R}f = \mathcal{R}(f-a) + a$ and $\mathcal{R}g = \mathcal{R}(g-a)+a$, and so
$$
\|\mathcal{R}f - \mathcal{R}g\|_\infty = \|\mathcal{R}(f-a) - \mathcal{R}(g-a)\|_\infty \leq \|(f-a) - (g-a)\|_\infty = \|f-g\|_\infty
$$
by the cited result for non-negative functions.

Sometimes it is desirable to construct an increasing (or non-decreasing) rearrangement of $f$ instead. For functions defined on the unit interval, this can easily be achieved by considering $t \mapsto \mathcal{R}f(1-t)$ instead of $\mathcal{R}f$ directly. If we view $f$ as a random variable defined on the Borel space $[0,1]$ equipped with the Lebesgue measure, and denote the distribution function of $f$ by $F$, then this definition agrees with the generalised inverse $F^{-1}$. More precisely, we have the identity
    $$
    F^{-1}(t) = \inf\left\{u \in \mathbb{R} ~|~ \lambda(f \leq u) \geq t \right\} = \inf\left\{u \in \mathbb{R} ~|~ \lambda(f > u) \leq 1-t \right\} = \mathcal{R}f(1-t)
    $$
for any $t \in (0,1)$. The first equality is simply the definition of the generalised inverse, and we have excluded the points $t = 0$ and $t = 1$ simply because the quantile function is often considered not defined in those points. If we extend the domain of $F^{-1}$ to $[0,1]$ by allowing $F^{-1}(0) = -\infty$ and $F^{-1}(1) = +\infty$, then the identity $F^{-1}(t) = \mathcal{R}f(1-t)$ holds for all $t \in [0,1]$.

\section{Rearranged density estimators}
\label{sec2}
In this section, we provide a discussion of rearranged density estimators. In order to give a clean exposition of the main findings without getting into technicalities that would obscure the main message, we focus on histogram based estimators in the case of the uniform density. Specifically, given the sample $U_1,\dots, U_n \sim$ i.i.d. $\Uc[0,1]$, fix $N<n$ and consider the histogram estimator 
\[
\hat f_{n,N}(x) := N\sum_{j=1}^{N} \textbf{1}\{x \in A_{j,N}\} (\hat F_n(j/N) - \hat F_n((j-1)/N))
\]
where $\hat F_n(u)={1\over n}\sum_{i=1}^n \textbf{1}\{ U_i \le u\} $ denotes the empirical CDF of $U_1,\dots,U_n$ and $A_{1,N} = [0,1/N]$, $A_{j,N} = ((j-1)/N,j/N], j=2,\dots,N$. Let $\mathcal{R} \hat f_{n,N}$ denote the decreasing rearrangement of the estimator $\hat f_{n,N}$ as defined in Section~\ref{sec:monotone_rearrangements}. Let $\ell^{\infty} (D)$ denote the space of bounded functions (defined on the domain $D$) and let the symbol $\weak$ denote weak convergence in the sense of Hoffman-J\o rgensen \citep[see][]{vandervaart_wellner:weak_convergence}. Our main result is as follows.

\begin{theorem}\label{thm:dens}
Assume that $U_1,\dots,U_n$ are i.i.d. $\Uc[0,1]$ and that $1 \ll N = o(\sqrt{n}/\log n)$. Then, for any fixed $\eps \in (0,1/2)$, 
\begin{align}\label{eq:procreardens}
\Big \{ \sqrt{n}\Big (\Rc \hat f_{n,N}(u) - 1 - \frac{N}{n} \Phi^{-1}(1-u)\Big )\Big  \}_{u \in [\eps,1-\eps]} \weak \G \quad \mbox{in } \ell^{\infty}[\eps,1-\eps]
\end{align}
where the limiting process is given by  $\G = \G_1 - Y$ with $\G_1$ a centered Gaussian process of the form $\G_1(u) = \B(1-u)/\varphi\circ\Phi^{-1}(1-u)$ for a standard Brownian bridge $\B$ and $Y \sim \mathcal{N}(0,1)$ such that $\G_1,Y$ are jointly normal with
\[
\E\big[\B(u)Y\big] = \E[Y\textbf{1}\{Y \le \Phi^{-1}(u)\}].
\]
\end{theorem}

Of course this theorem can also be applied to monotonically increasing estimators by considering $\mathcal{R}\hat{f}_{n,N}(1-u)$ instead of $\mathcal{R}\hat{f}_{n,N}(u)$. The limiting results in Theorem~\ref{thm:dens} are strikingly different from results for rearrangements of density estimators when the true density is strictly monotone, see \cite{AnevskiFougeres2019} and the references cited therein.

First, the stochastic part of the problem converges at a $n^{-1/2}$ rate, irrespective of the bandwidth choice. In the classical case, one obtains an $(N/n)^{1/2}$ rate for the stochastic part with an additional bias contribution. Second, the estimators also converge as processes. In the classical case this is not possible as the properly standardized  finite-dimensional distributions of $\Rc \hat f_{n,N}(u) - f(u)$ converge to collections of independent Gaussians. Finally, we note that extending the process convergence in~\eqref{eq:procreardens} to the full interval $(0,1)$ is not possible. Indeed, it is straightforward to see that the variance of the limit $\G(u)$ explodes as $u \to 0$ or $u \to 1$. In fact, one can show that, assuming $N = o(n/(\log n)^3)$, 
\begin{align}\label{eq:Gumbel}
a_N^{-1}\Big(\sup_{u \in (0,1)} \sqrt{n/N}(\Rc \hat f_{n,N}(u) - 1) - b_N \Big) \weak G     
\end{align}
for a standard Gumbel random variable $G$, where
$$
a_N = 1/\sqrt{2\log N}~, ~~b_N = \sqrt{2\log N} - \frac{\log\log N +\log(4\pi)}{2\sqrt{2\log N}}.
$$
Results of this type are well known in kernel density estimation \citep[see, for example][]{BickelRosenblatt1973,KonakovPiterbarg1982,HallTitterington1988} and we provide a short self-contained proof for the rearranged histogram estimator below.

Before proceeding to the proof, we note that similar results can be expected to hold for rearrangements of kernel density estimators and for rearrangements of local smoothing estimators of regression functions, such as the Nadaraya-Watson or local polynomial estimators. Since no interesting new insights are likely to be gained in those cases, we omit further discussions for the sake of brevity.   

Below, we provide a complete proof of Theorem~\ref{thm:dens} since it is relatively short and illustrative of the main points. The corresponding proof for rearranged copulas given in the next section is substantially more technical and some of the core structure is obscured. 

\begin{proof}[Proof of Theorem~\ref{thm:dens} and equation~\eqref{eq:Gumbel} ]

Let $G_n = \sqrt{n}(\hat{F}_n - F)$ denote the empirical process based on $U_1, \ldots, U_n$ ($F$ is the distribution function of a $\mathcal{U}[0,1]$-distribution). For $x \in A_{j,N}$ we have
\begin{equation}
    \label{eq:kmt_density}
\hat{f}_{n,N}(x) - 1 = N\left[\hat{F}_n\left(\frac{j}{N}\right) - \frac{j}{N} - \hat{F}_n\left(\frac{j-1}{N}\right) + \frac{j-1}{N}\right] = \frac{N}{\sqrt{n}}\left[G_n\left(\frac{j}{N}\right) - G_n\left(\frac{j-1}{N}\right)\right].
\end{equation}
By the Hungarian construction \citep[Theorem 3 in Chapter 12, Section 1, of][]{shorack_wellner:1986}, there exists a triangular array of row-wise independent $\Uc[0,1]$ random variables $U_{in}, i \in [n], n \in \N$ and a sequence of standard Brownian bridges $\B_n, n \in \N$ such that, denoting by $\tilde{G}_n$ the empirical process of $U_{in}, i \in [n]$, we have 
$$
\sup_{0 \leq x \leq 1}\left|\tilde{G}_n(x) - B_n(x)\right| = \mathcal{O}_{\P}\left(\frac{\log n}{\sqrt{n}}\right).
$$
Since each $\tilde G_n$ has the same distribution as $G_n$ and since our main result is only concerned with convergence in distribution, we will henceforth work with $\tilde G_n$ in place of $G_n$ and define 
\[
\widetilde{f}_{n,N}(x) := 1 + \frac{N}{\sqrt{n}} \sum_{j=1}^N \textbf{1}\{x \in A_{j,N}\}\left[G_n\left(\frac{j}{N}\right) - G_n\left(\frac{j-1}{N}\right)\right]
\]
By Eq.\@ \eqref{eq:kmt_density}, we therefore get
\begin{equation}
\label{eq:R_dist_1}
\widetilde{f}_{n,N}(x) - 1 = \frac{N}{\sqrt{n}}\left[B_n\left(\frac{\lfloor xN\rfloor}{N}\right) - B_n\left(\frac{\lfloor xN \rfloor-1}{N}\right)\right] + \frac{N}{\sqrt{n}}R_n(x)
\end{equation}
in distribution as processes indexed in $x \in [0,1]$, where $B_n$ are Brownian bridges and $R_n$ a remainder term satisfying $\sup_{0 \leq x \leq 1}|R_n(x)| = \Oc_\P( n^{-1}\log n)$. Write
\begin{equation}
\label{eq:R_dist_2}
g_{n}(x) := \frac{N}{\sqrt{n}} b_{n}(x) + \frac{N}{\sqrt{n}}R_n(x) :=\frac{N}{\sqrt{n}}\left[B_n\left(\frac{\lfloor xN\rfloor}{N}\right) - B_n\left(\frac{\lfloor xN \rfloor-1}{N}\right)\right] + \frac{N}{\sqrt{n}}R_n(x).
\end{equation}
Let us consider the rearrangement of $g_n$. By Theorem 1 in \cite{AnevskiFougeres2019}, we have
\begin{equation}
    \label{eq:R_dist_3}
\frac{N}{\sqrt{n}}\|\Rc g_n - \Rc b_n\|_\infty \leq \frac{N}{\sqrt{n}}\|g_n - b_n\|_\infty = \frac{N}{\sqrt{n}} \|R_n\|_\infty = \mathcal{O}_{\P}\left(\frac{N \log n}{n}\right).
\end{equation}
Thus 
\[
\mathcal{R}\widetilde{f}_{n,N} - 1 = \mathcal{R}g_n = \frac{N}{\sqrt{n}}\mathcal{R}b_n + \mathcal{O}_{\P}\left(\frac{N \log n}{n}\right) \quad \mbox{in } \ell^\infty[0,1].
\]
Since we assumed that $N\log n/\sqrt{n} = o(1)$, the claim in~\eqref{eq:procreardens} follows after a trivial change of variable once we prove that 
\begin{align}\label{eq:procreardens-interm}
\sqrt{n}\Big(\frac{N}{\sqrt{n}}\mathcal{R}b_n(u) - \frac{N}{n} \Phi^{-1}(1-u)\Big)_{u \in [\varepsilon, 1-\varepsilon]} \weak \G \quad \mbox{in } \ell^{\infty}[\eps,1-\eps].
\end{align}
Since $B_n(x) \stackrel{\Dc}{=} W(x) - x W(1)$ as processes on $[0,1]$ for a standard Brownian motion $W$, we get
\begin{align}
\begin{split}
\label{eq:R_dist_4}
b_n(x) &= B_n\left(\frac{\lfloor xN \rfloor}{N}\right) - B_n\left(\frac{\lfloor xN \rfloor - 1}{N}\right) \\
&\stackrel{\Dc}{=} W\left(\frac{\lfloor xN \rfloor}{N}\right) - \frac{\lfloor xN\rfloor}{N}W(1)  - W\left(\frac{\lfloor xN \rfloor - 1}{N}\right) + \frac{\lfloor xN\rfloor - 1}{N}W(1) \\
&= W\left(\frac{\lfloor xN \rfloor}{N}\right) - W\left(\frac{\lfloor xN \rfloor - 1}{N}\right) - \frac{1}{N} W(1) \\
&= \Delta_{\lfloor xN\rfloor, N} - \frac{1}{N} \sum_{k=1}^N \Delta_{k,N}
\end{split}
\end{align}
in distribution simultaneously for all $x \in [0,1]$, where $\Delta_{k,N} = W(k/N) - W([k-1]/N)$, $k = 1, \ldots, N$. Since Brownian motion has independent increments, the $\Delta_{1,N}, \ldots, \Delta_{N,N}$ are row-wise i.i.d. \@ $\mathcal{N}(0,1/N)$-distributed. Write $\overline{\Delta}_N = (\Delta_{1,N} + \ldots + \Delta_{N,N})/N$ and $Z_{k,N} = \frac{N}{\sqrt{n}}\{\Delta_{k,N} - \overline \Delta_N\}, k \in [N]$, and denote by $\widehat F_Z$ the empirical distribution function of $Z_{1,N},\dots,Z_{N,N}$. Then we have 
\[
\left(\frac{N}{\sqrt{n}}\mathcal{R}b_n(1-u)\right)_{u \in [\varepsilon, 1-\varepsilon]} \stackrel{\Dc}{=} \left(\hat{F}_Z^{-1}(u)\right)_{u \in [\varepsilon, 1-\varepsilon]} . 
\]
Let $Y_{i,N} := \sqrt{N} \Delta_{i,N}, i \in [N]$ and denote by $\widehat F_Y$ the empirical distribution function of $Y_{1,N},\dots, Y_{N,N}$ and by $\widehat F_{\Delta}$ the empirical distribution function of $\Delta_{1,N},\dots,\Delta_{N,N}$. From the definition of $Z_{k,N}$ we have 
\begin{equation}
    \label{eq:bahadur_prep}
\widehat F_{Z}^{-1}(u) = \frac{N}{\sqrt{n}} 
\widehat F_{\Delta}^{-1}(u) - \frac{N}{\sqrt{n}}  \overline\Delta = \sqrt{\frac{N}{n}}\Big( \widehat F_{Y}^{-1}(u) - \overline Y\Big). 
\end{equation}
Recall the following consequence of the classical Bahadur-Kiefer theorems \citep[e.g.\@ Theorem 2 in Chapter 18, Section 2, of][]{shorack_wellner:1986}. If $\zeta_1, \ldots,  \zeta_N$ are i.i.d.\@ $\mathcal{N}(0,1)$-random variables with empirical distribution function $G_N$, then
$$
\sup_{\varepsilon \leq u \leq 1-\varepsilon} \left|G_N^{-1}(u) - \Phi^{-1}(u) + \frac{G_N \circ \Phi^{-1}(u) - u}{\phi \circ \Phi^{-1}(u)}\right| = \mathcal{O}_{a.s.}\left(\frac{(\log\log N)^{1/4} (\log N)^{1/2}}{N^{3/4}}\right).
$$
In particular, the difference on the left-hand side is $o_\mathbb{P}(N^{-1/2})$. This last $o_\mathbb{P}$ statement is purely distributional, and so it remains valid if we replace the i.i.d.\@ sample $\zeta_1, \ldots, \zeta_N$ with $Y_{1,N}, \ldots, Y_{N,N}$, as their joint distribution is equal to that of $\zeta_1, \ldots, \zeta_N$ for any $N$. Thus,
\[
\widehat F_{Y}^{-1}(u) - \Phi^{-1}(u) = - \frac{\widehat F_Y\circ \Phi^{-1}(u) - u}{\varphi\circ\Phi^{-1}(u)} + o_{\P}(N^{-1/2})
\]
uniformly on $[\eps,1-\eps]$. In summary, we obtain
\begin{align*}
&\widehat F_{Z}^{-1}(u) - \sqrt{\frac Nn}\Phi^{-1}(u) 
 = - \frac{1}{\varphi\circ\Phi^{-1}(u)} \sqrt{\frac Nn}\Big(\sum_{i=1}^N \textbf{1}\{Y_i \le \Phi^{-1}(u)\} - u \Big) - \sqrt{\frac Nn} \overline Y + o_\P(n^{-1/2})
\end{align*}
uniformly on $[\eps,1-\eps]$. The processes $\G_{1,n}(u) := \sqrt{N}(\widehat F_Y\circ \Phi^{-1}(1-u) - (1-u))$ and $\G_{2,n}(u) := \sqrt{N} \times \overline Y$ are jointly tight in $\ell^\infty([\eps,1-\eps])$ since each process is tight, and so they converge jointly in $\ell^\infty([\eps,1-\eps])$ since their finite-dimensional distributions converge by the multivariate central limit theorem. The continuous mapping theorem thus yields
\[
\sqrt{n}\Big( \widehat F_{Z}^{-1} - \sqrt{\frac Nn}\Phi^{-1}  \Big) \rightsquigarrow \G(1-\cdot)
\]
as process in $\ell^\infty([\eps,1-\eps])$ for the process in~\eqref{eq:procreardens}. This proves~\eqref{eq:procreardens-interm} and completes the proof of Theorem~\ref{thm:dens}.

For a proof of~\eqref{eq:Gumbel}, a combination of \eqref{eq:R_dist_1}, \eqref{eq:R_dist_2}, \eqref{eq:R_dist_3}, \eqref{eq:bahadur_prep} we have $\sup_{x \in (0,1)} \Rc\hat f_{n,k}(x) \stackrel{\Dc}{=} \sup_{x \in (0,1)} \Rc\widetilde{f}_{n,k}(x)$ and
\begin{equation}
    \label{eq:marius_doesnt_get_it}
\sup_{x \in (0,1)} \Rc\widetilde{f}_{n,k}(x) -1 = \sqrt{\frac{N}{n}} \max_{i=1, \ldots, N} \{Y_i - \overline Y\} + \Oc_{\P}\left(\frac{N \log n}{n}\right).
\end{equation}
Since $\sqrt{N}\times\overline Y = \Oc_\P(1)$,~\eqref{eq:Gumbel} follows by standard results about maxima of independent standard normal random variables \citep[e.g.\@ Example 21.16 in][]{vandervaart:asymptotic_statistics}.
\end{proof}

\section{Rearranged copulas and copula dependence measures and their estimators}
\label{sec3}
For any copula $C : [0,1]^2 \to [0,1]$, we define the
(stochastically increasing) rearranged copula $C^\uparrow$ by
\begin{align}
\label{det0}
  C^\uparrow(u,v) := \int_0^u \mathcal{R}_1(\partial_1 C)(s,v) ~\mathrm{d}s,
\end{align}
where $\partial_1 C(s,v) := \frac{\mathrm{d}}{\mathrm{d}s} C(s,v)$ denotes the partial derivative with respect to the first coordinate, and $\mathcal{R}_1f(s,t)$ is the decreasing rearrangement as defined in \eqref{eq:decreasing_rearrangement_definition} of a function $f : [0,1]^2 \to \mathbb{R}$ with respect to the first coordinate, more precisely $\mathcal{R}_1f(s,t) := [\Rc f(\cdot,t)](s)$. \cite{Strothmann.2022} showed that  the rearranged copula is {\it stochastically increasing}, which means that the function $u\mapsto \partial_1 C^\uparrow(u,v)$ is decreasing for each $v$. 

In the following we consider two estimators  for rearranged dependence measures from  a sample   $(X_1, Y_1), \ldots, (X_n, Y_n)$  of independent identically distributed observations   with copula $C$.  The first  one is based on the empirical copula \citep[see, for example,][]{Ruschendorf,wegkamp2004,Tsukahara,segers:2012}. It is obtained by using the identity $C(u,v) = F(F_X^{-1}(u), F_Y^{-1}(v))$, where $F$ and $F_X, F_Y$ denote the joint and marginal distribution functions of $X$ and $Y$, respectively, and then replacing the distribution functions $F$, $F_X$ and $F_Y$ in this representation by their corresponding estimates $F_n (x,y) = {1\over n } \sum_{i=1}^n \textbf{1}\{ X_i \leq x  , Y_i \leq y \}  $, $F_{X,n} (x) ={1\over n }\sum_{i=1}^n \textbf{1}\{ X_i \leq x \} $ and $F_{Y,n}(y) = {1\over n }\sum_{i=1}^n \textbf{1} \{  Y_i \leq y \} $. This results in the empirical copula 
$$
C_n(u,v) = F_n(X_{(\lceil un \rceil)}, Y_{(\lceil vn \rceil)}) , 
$$
where  $X_{(1)} \leq \ldots \leq X_{(n)}$  and $Y_{(1)} \leq \ldots \leq Y_{(n)}$ denote the order statistics of $X_{1} ,  \ldots , X_{n} $ and $Y_{1} ,  \ldots , Y_{n} $, respectively. The partial derivative $\partial_1 C(u,v)$ is then estimated by considering the increments of $C_n(u,v)$ in $u$ with respect to some step size $1/N_1$. More precisely, if $u \in \big [{k-1 \over N_1}, {k \over N_1} \big )$, we set
$$
  \widehat{\partial_1 C}(u,v) = N_1 \big(C_n(k/N_1, v) - C_n((k-1)/N_1, v)\big), \quad  k=1,\dots, N_1,
$$
where $N_1 \in \mathbb{N}$. Finally, the estimator of the rearranged copula   $C^\uparrow$ is defined by
\begin{align}
    \label{det1}
  \widehat{C^\uparrow}(t,v) := \int_0^t \mathcal{R}_1\left(\widehat{\partial_1 C}\right)(s,v) ~\mathrm{d}s, 
\end{align}

 The second estimator is based on the checkerboard copula which is  an important tool in statistical applications 
 \citep[see][and the references therein for a detailed discussion of this estimator]{GENEST201782,Junker.2021}. 
We define  for given bandwiths $N_1, N_2$, both dependent on $n$,  the checkerboard copula $C_n^\#$ by linearly interpolating the values of the empirical copula $C_n$. More precisely, we define $C_n^\#(j/N_1, k/N_2) = C_n(j/N_1, k/N_2)$ for all $j = 0, \ldots, N_1$ and $k = 0, \ldots, N_2$, and interpolate these values such that the restriction of $C_n^\#$ to each $(N_1^{-1} [j, j+1]) \times (N_2^{-1} [k, k+1])$, $j = 0, \ldots, N_1 - 1$, $k = 0, \ldots, N_2 - 1$, is linear in both arguments \citep[for a precise definition, see for instance Section 3 in][]{Strothmann.2022}. With these notations we define by 
\begin{align}
    \label{det3}
  \widehat{C^{\# \uparrow}}(t,v) = \int_0^t \mathcal{R}_1\left(\partial_1 C_n^\#\right) (s,v
  ) ~\mathrm{d}s.
\end{align}
the checkerboard-based estimator for rearranged copula $C^\uparrow$. The consistency $\widehat{C^{\# \uparrow}}$ for $C^{\uparrow}$ with respect to the metrics
$$
D_p(C_1,C_2) := 
\Big(\int_0^1 \int_0^1
\left|
\partial_1 C_1(u,v)-\partial_1 C_2(u,v)
\right|^p
\,du\,dv\Big)^{1/p}, \quad 1 \le p < \infty
$$
follows from the arguments  in the proof of Theorem 3.4 in \cite{Strothmann.2022} and the fact that $| \partial_1 C_1(u,v)-\partial_1 C_2(u,v) |^p
\leq |
\partial_1 C_1(u,v)-\partial_1 C_2(u,v) |$ \citep[note that $\partial_1 C(u,v)=
\Pr(V\leq v\mid U=u)$, where the cdf of the vector $(U,V)$ is the copula $C$; see Theorem 2.2.7 in][]{nelsen:copulas}. However, to the best of our knowledge, no results on the weak convergence of the process $\widehat{C^{\# \uparrow}}$ or $\widehat{C^\uparrow}$, or on the resulting dependence measures are available in the literature so far.

\subsection{Weak convergence of rearranged copula estimators under independence}
\label{sec:weak_convergence}

In this section we establish weak convergence of the appropriately centered and normalized estimators $\widehat{C^\uparrow}$ and $\widehat{C^{\# \uparrow}}$ to Gaussian processes assuming that the underlying truth $C$ corresponds to the independence copula. The independence copula is practically important because it arises as the null model when testing independence. It is also the most challenging one to analyze asymptotically due to the difficulties of analyzing rearrangements of flat functions.   

\begin{theorem}
  \label{thm:asymptotik}
  Assume that the random variables $X,Y$ are independent with continuous marginal distribution functions. Let $N_1 = N_1(n)$ satisfy $N_1 \to \infty$ for $n \to \infty$ and $N_1 (\log n)^{3/2} = o\left(n^{1/4}\right)$. Then we have for the rearranged copula process defined in \eqref{det1}
  $$
   \bigg  \{  \sqrt{n}\bigg (\widehat{C^\uparrow}(t,v) - tv - \sqrt{\frac{N_1}{n}} \int_0^t \Phi_{v(1-v)}^{-1}(1-u) ~\mathrm{d}u\bigg )\bigg \}_{(v,t) \in [0,1]^2} \rightsquigarrow \bigg \{ \int_0^t H(u,v) ~\mathrm{d}u \bigg \} _{(v,t) \in [0,1]^2},
  $$
  where
  $$
    H(u,v) = \frac{G(1-u,v) }{\varphi_{v(1-v)}\left(\Phi_{v(1-v)}^{-1}(1-u)\right)} - G_1(v),
  $$
  $G_1$ is a standard Brownian bridge, $G$ is a centered Gaussian process and $\varphi_{v(1-v)}$ and $\Phi_{v(1-v)}$ denote the density and cumulative distribution function, respectively, of a centered normal distribution with variance $v(1-v)$. The covariance structure of $G$ and $G_1$ is given by
  $$
    \mathrm{Cov}(G(u,v), G(s,t)) = \mathrm{Cov}\left(\textbf{1}\left(B(v) \leq \Phi^{-1}_{v(1-v)}(u)\right), \textbf{1}\left(B(t) \leq \Phi_{t(1-t)}^{-1}(s)\right)\right)
  $$
  and
  \begin{equation}
    \label{eq:kovarianzfunktion_2}
    \mathrm{Cov}(G(u,v), G_1(w)) = \mathbb{E}\left[B(w) \textbf{1}\left( B(v) \leq \Phi_{v(1-v)}^{-1}(u)\right)\right]
  \end{equation}
  for a standard Brownian bridge $B$.
\end{theorem}
Note that 
\begin{align} \label{center1}
\int_0^t \Phi_{v(1-v)}^{-1}(1-u) ~\mathrm{d}u = \sqrt{v(1-v)} \int_0^t \Phi^{-1}(1-u) ~\mathrm{d}u
= \sqrt{v(1-v)} \varphi\bigl(\Phi^{-1}(1-t)\bigr),
\end{align}
where $\varphi =\Phi^\prime $ is the density of the standard normal distribution. 
We have used the fact $\Phi_a^{-1}(t) = \sqrt{a} \Phi^{-1}(t)$ $(a> 0) $ and the last equality follows from the substitution \(z=\Phi^{-1}(1-u)\) and the convention that all expression are $0$ if  
\(v=0\) or \(v=1\).

In contrast to the density case, we can obtain convergence as a process of the entire set $[0,1]^2$. This is possible since the rearranged density estimator is integrated. Thus the divergence of the rearranged process close to the boundary points can be alleviated as it only happens on sets of small measure. Proving that the divergence rate and measure of the set where divergence happens have the correct interplay is one of the many technical challenges in the proof. An additional challenge is that, in contrast to the density case, we need to
establish tightness of the process in two arguments. 
The complete proof of Theorem~\ref{thm:asymptotik} is very technical and subtle. The main steps are discussed in Section~\ref{sec:proof_outline}, with several intermediate technical results deferred to Appendix~\ref{sec:proof-cop}. 

\smallskip

The next result shows that under suitable conditions on the bandwidths, the rearranged checkerboard copula estimator is asymptotically equivalent to to the estimator above.

\begin{theorem}
\label{thm:checkerboard}
Assume that the random vector $(X,Y)$ has continuous marginal distribution functions. Let that $N_1$ and $N_2$ be such that $N_1, N_2 \to \infty$ and  $N_1/N_2 + N_1/n = o(1/\sqrt{n})$ for $n \to \infty$. Then we have
$$
\sup_{0 \leq t,v \leq 1}\sqrt{n}\left|\widehat{C^{\# \uparrow}}(t,v) - \widehat{C^{\uparrow}}(t,v)\right| \xrightarrow[n \to \infty]{a.s.} 0,
$$
where $\widehat{C^{\# \uparrow}}$ is the checkerboard-based estimator defined by \eqref{det3}. 
\end{theorem}

\subsection{Weak convergence of estimators of rearranged dependence measures}

For a given copula-based dependence measure $\mu $,  \cite{Strothmann.2022} defined a  
 \emph{rearranged dependence measure} by  
\begin{equation}
    \label{hd1}
R_{\mu}(C) := \mu(\ir{C}) .
\end{equation} 
They then showed, that, 
if a given dependence measure  $\mu$  satisfies the properties (1)--(3) on the set of all stochastically increasing copulas, its rearranged measure $R_\mu$ satisfies these properties on the set of all copulas. Prominent examples, where this statement is applicable are Spearman's $\rho$ and Kendall's $\tau$, and it therefore is possible to construct   a rearranged Spearman's $\rho$ or Kendall's $\tau$ from those classical measures, which satisfy the conditions (1) - (3). Another class of measures, which can be rearranged such that the new dependence measures satisfy (1) - (3)  are the  Schweizer-Wolff measures \citep{Schweizer.1981}.

The resulting estimators of the rearranged dependence measure \eqref{hd1} are given by
\begin{align}
    \label{hd2}
\hat{R}_{\mu} &:= \mu(\widehat{\ir{C}}) ,\\ 
    \label{hd3}
 \hat{R}_{\mu}^\# &:= \mu(\widehat{C^{\# \uparrow}}). 
\end{align}
Some care is needed when considering $\hat{R}_\mu$, since the empirical copula $C_n$ and its rearrangement $\widehat{C^\uparrow}$ are generally not copulas. In contrast, since the checkerboard copula $C_n^\#$ is in fact a copula, so is its rearrangement $\widehat{C^{\#\uparrow}}$ \citep[by Theorem 2.3 in][]{Strothmann.2022}. To avoid these technical issues, we will focus our discussion on the checkerboard-based estimator $\hat{R}_\mu^\#$, noting that similar results can be derived for $\hat{R}_\mu$ if this expression is well-defined.

The preceding results imply asymptotic normality of $\hat{R}_\mu^\#$ by virtue of the Delta method. More specifically, we define the processes $T_n$ by
\begin{align}
\label{center}
T_n(t,v) = tv + \sqrt{\frac{N_1}{n}} \sqrt{v(1-v)}~\varphi\left(\Phi^{-1}(1-t)\right)
\end{align}
and define, for a given $\kappa > 0$, the subset $\mathcal{C} = \mathcal{C}(\kappa) \subseteq \ell^\infty([0,1]^2)$ to consist of all functions of the form
$$
(t,v) \mapsto C(t,v) + \varepsilon \sqrt{v(1-v)} \varphi\left(\Phi^{-1}(1-t)\right),
$$
where $C$ is a copula and $\varepsilon \in [0,\kappa]$. The set $\mathcal{C}$ contains all copulas as well as almost all $T_n$. We assume that  the dependence measure $\mu (C)$ is defined for every $C \in \mathcal{C}$ and  understand $\mu : C \mapsto \mu(C)$ as a map on $\mathcal{C}$. The following corollary is an immediate consequence of the functional Delta-method \citep[Theorem 3.9.5 in][]{vandervaart_wellner:weak_convergence}. The required differentiability assumption is stronger than regular tangential Hadamard differentiability, due to the presence of the bias terms $T_n$. It is weaker than uniform Hadamard differentiability since the latter allows for more general sequences instead of $T_n$.
\begin{corollary}
\label{cor:delta_method}
    Suppose that the assumptions of Theorems \ref{thm:asymptotik} and \ref{thm:checkerboard} are satisfied and that there exists a function $\mu' : \ell^\infty([0,1]^2) \to \mathbb{R}$ such that
    $$
    \sqrt{n}\left[\mu(T_n + n^{-1/2} h_n) - \mu(T_n)\right] \xrightarrow[n \to \infty]{} \mu'(h)
    $$
    for every uniformly convergent sequence $(h_n)_{n \in \mathbb{N}}$ in $  \ell^\infty([0,1]^2)$ with $T_n + n^{-1/2}h_n \in \mathcal{C}$, where $h = \lim_{n \to \infty}  h_n$. If $X$ and $Y$ are independent, then
    $$
    \sqrt{n}\left(\hat{R}_\mu^\# - \mu(T_n)\right) \weak \mu'(L),
    $$
    where $ L(t,v)=\int_0^t H(u,v)\,\mathrm{d}u$  is the limiting process from Theorem \ref{thm:asymptotik}. Consequently, if $\mu'$ is linear and continuous on $\ell^\infty([0,1]^2)$ with respect to the supremum norm, then
        $$
    \sqrt{n}\left(\hat{R}_\mu^\# - \mu(T_n)\right) \weak \mathcal{N}\left(0, \sigma^2\right)
    $$
    for some $\sigma^2 \geq 0$ depending on the function $\mu'$.
\end{corollary}

\begin{example} {\rm  A popular dependence measure is Spearman's $\rho$, which defined by
\[
    \rho(C)   =  12\int_0^1\int_0^1 C(t,v)\,\mathrm{d}t\,\mathrm{d}v-3,
\]
and it was shown in \cite{Strothmann.2022} that  the corresponding rearranged dependence measure $R_\rho(C)=\rho(C^\uparrow)$ has the properties  (1) - (3) mentioned in the introduction. 

As the copula dependent part of $\rho$ is linear, the differentiability condition in Corollary~\ref{cor:delta_method} is obviously satisfied with
\[
    \rho'(h)
    =
    12\int_0^1\int_0^1 h(t,v)\,\mathrm{d}t\,\mathrm{d}v,
    \qquad
    h\in\ell^\infty([0,1]^2).
\]
In particular, \(\rho'\) is linear and continuous on
\(\ell^\infty([0,1]^2)\) with respect to the supremum norm.
Consequently, under the assumptions of Theorem~\ref{thm:asymptotik} and independence of
\(X\) and \(Y\), Corollary~\ref{cor:delta_method} yields
$$
    \sqrt{n}\bigl( \hat{R}_\rho^\# -\rho(T_n)\bigr) \xrightarrow{\mathcal D}
    N(0,\sigma_\rho^2).
$$
Observing the definition of $T_n$ in~\eqref{center} and $\rho$ we obtain for the centering term
\begin{align*}
    \rho(T_n)
    &=
    12\sqrt{\frac{N_1}{n}}
    \int_0^1\int_0^1\int_0^t
    \Phi_{v(1-v)}^{-1}(1-u)
    \,\mathrm{d}u\,\mathrm{d}t\,\mathrm{d}v \\
    &= 12\sqrt{\frac{N_1}{n}}
    \Big \{
        \int_0^1\sqrt{v(1-v)}\,\mathrm{d}v
    \Big \}
    \Big\{  \int_0^1 \varphi \big (\Phi^{-1}(1-t) \big ) \,\mathrm{d} t  \Big \} .
\end{align*}
where we have used  the identity  \eqref{center1}. 
The first integral is a Beta integral and given by  
$ \int_0^1\sqrt{v(1-v)}\,\mathrm{d}v
    =\frac{\pi}{8} . 
    $
    For the second integral the substitution $z = \Phi^{-1}(1-t)$ yields 
\[
\int_0^1 \varphi \big ( \Phi^{-1}(1-t) \big ) \mathrm{d}t = \int_{-\infty}^{\infty} \varphi^2(z) \mathrm{d}z =  \frac{1}{2\sqrt{\pi}} , 
\]
which  gives for the centering term 
\[
    \rho (T_n)
    =   \frac{3\sqrt{\pi}}{4}
    \sqrt{\frac{N_1}{n}}. 
    \]
For the calculation of the variance note that 
\[
   \int_0^1  L(t,v) \mathrm{d} t  =  \int_0^1  \int_0^t H(u,v)\,\mathrm{d}u \, \mathrm{d} t  =   \int_0^1 (1-u) H(u,v) \,\mathrm{d}u .
  \]
which gives for the asymptotic variance of the rearranged Spearman's $\rho$ estimator
\[
\begin{aligned}
    \sigma_\rho^2
    &=
    144
    \int_0^1\int_0^1\int_0^1\int_0^1
    {\rm Cov}\bigl(L(t,v),L(s,w)\bigr)
    \,\mathrm{d}t\,\mathrm{d}v\,\mathrm{d}s\,\mathrm{d}w
    \\
    &=
    144
    \int_0^1\int_0^1\int_0^1\int_0^1
    (1-u)(1-r)
    {\rm Cov}\bigl(H(u,v),H(r,w)\bigr)
    \,\mathrm{d}u\,\mathrm{d}v\,\mathrm{d}r\,\mathrm{d}w .
\end{aligned}
\]
}
\end{example}

\section{Proof of Theorem \ref{thm:asymptotik} and Theorem \ref{thm:checkerboard}}
\label{sec:proof_outline}

\subsection{Proof of Theorem~\ref{thm:asymptotik}}

In this section we provide the main proof structure and key steps. Intermediate technical results which are cited throughout this proof are stated and proved in detail in Appendix~\ref{sec:proof-cop}. Since we only use one bandwidth, we will write $N$ instead of $N_1$ throughout this section to simplify notation.

\smallskip
 
Writing $G_n(u,v) = \sqrt{n}(C_n(u,v) - C(u,v))$, observe that
$$
  \widehat{\partial_1 C}(u,v) = \frac{N}{\sqrt{n}}(G_n(k/N, v) - G_n((k-1)/N, v)) + N(C(k/N, v) - C((k-1)/N, v))
$$
for $k-1 \leq uN < k$, $k = 1, \ldots, N$. Under the assumption of independence we have $C(u,v) = uv$, i.e.\@ $N(C(k/N, v) - C((k-1)/N, v)) = v$, and so rearranging the above equality yields
$$
  \sqrt{n}\left(\widehat{\partial_1 C}(u,v) - v\right) = N(G_n(k/N, v) - G_n((k-1)/N, v)).
$$
By Proposition 4.2 in \cite{segers:2012}, we have 
\[
\sup_{u,v\in [0,1]} \big|G_n(u,v) - \widetilde{G}_n(u,v)\big| = o_{a.s.}(n^{-1/4}\log^{3/4} n)
\]
where 
\[
\widetilde{G}_n(u,v) := G^\circ_n(u,v) - vG^\circ_n(u,1) - uG^\circ_n(1,v) 
\]
and $G^\circ_n$ is the empirical process of $(F_1(X_i),F_2(Y_i))_{i \in [n]}$. This implies
\[
\widehat{\partial_1 C}(u,v) = \widetilde{\partial_1 C}(u,v)  + R_{n,1}(u,v)
\]
where 
\[
\widetilde{\partial_1 C}(u,v) := \frac{N}{\sqrt{n}}\big(\widetilde G_n(k/N, v) - \widetilde G_n((k-1)/N, v)\big) + N\big(C(k/N, v) - C((k-1)/N, v)\big)
\]
and
\[
\sup_{u,v\in [0,1]} \big|R_{n,1}(u,v)\big| = o_{a.s.}(Nn^{-3/4}\log^{3/4} n).
\]
By Theorem 1 in \cite{massart1989strong} we can enlarge the probability space and on that enlarged space construct a sequence of centered Gaussian processes $K_n$ with covariance structure $\E[K_n(R)K_n(R')] = \lambda(R \cap R') - \lambda(R) \lambda(R')$ for two rectangles $R,R'$ of the form $[0,u] \times [0,v]$ and bivariate Lebesgue measure $\lambda$ such that
\[
\sup_{u,v\in [0,1]} \big|G_n^\circ(u,v) - K_n(u,v)\big| = \Oc_{\P}(n^{-1/4}\log^{3/2} n). 
\]
Thus, on the enlarged probability space, we obtain 
\begin{equation}\label{eq:linpartialC1}
\widehat{\partial_1 C}(u,v) = \overline{\partial_1 C}(u,v) 
+ R_{n,2}(u,v)
\end{equation}
where 
\begin{align*}
\overline{\partial_1 C}(u,v) &:= \frac{N}{\sqrt{n}}\big(\widetilde K_n(k/N, v) - \widetilde K_n((k-1)/N, v)\big) + N\big(C(k/N, v) - C((k-1)/N, v)\big) 
\\
\widetilde K_n(u,v) &:= K_n(u,v) - vK_n(u,1) - uK_n(1,v) 
\end{align*}
and
\[
\sup_{u,v\in [0,1]} \big|R_{n,2}(u,v)\big| = \Oc_{\P}(Nn^{-3/4}\log^{3/2} n).
\]
Next observe that for an arbitrary bounded function $f: [0,1]^2 \to \R$, with a slight abuse of notation
$$
\mathcal{R}_1\left[\sqrt{n}\left(f(u,v) - v\right)\right] = \sqrt{n}\left(\mathcal{R}_1f (u,v) - v\right).
$$
By Theorem 1 in \cite{AnevskiFougeres2019}, we thus have, again with a slight abuse of notation
\begin{equation}\label{eq:linpartialC2}
\sup_{u,v \in [0,1]}\left|\mathcal{R}_1\left[\sqrt{n}\left(\widehat{\partial_1 C}(u,v) - v\right)\right] - \sqrt{n}\left( \mathcal{R}_1\left(\overline{\partial_1 C}\right)(u,v) - v\right)\right| = o_{a.s.}(Nn^{-1/4}\log^{3/4} n) = o_{a.s}(1)
\end{equation}
by our assumptions on $N$, which also implies that
$$
\sqrt{n}\left(\widehat{C^\uparrow}(t,v) - tv\right) = \sqrt{n}\left(\int_0^t \mathcal{R}_1\left(\overline{\partial_1 C}\right)(u,v) ~\mathrm{d}u - tv\right) + o_{a.s.}(1).
$$
It therefore suffices to prove
\begin{align}
\begin{split}
\label{eq:pfouthelp1}
&\sqrt{n}\left(\int_0^t \mathcal{R}_1\left(\overline{\partial_1 C}\right)(u,v) ~\mathrm{d}u - tv - \sqrt{\frac{N}{n}} \int_0^t \Phi_{v(1-v)}^{-1}(1-u) ~\mathrm{d}u\right)_{(v,t) \in [0,1]^2} \\
&\rightsquigarrow \left(\int_0^t H(u,v) ~\mathrm{d}u\right)_{(v,t) \in [0,1]^2}.
\end{split}
\end{align}

Writing
\begin{equation}\label{eq:defDC}
  D(k,v) = K_n(k/N,v) - K_n((k-1)/N,v) - v \big[K_n(k/N,1) - K_n((k-1)/N,1)\big],
\end{equation}
we see that the process $(u,v) \mapsto \sqrt{n}\left(\overline{\partial_1 C}(u,v) - v\right)$ is equal in distribution to
\begin{equation}
  \label{eq:first_step}
  (u,v) \mapsto N D(k,v) - K_n(1,v) 
\end{equation}
where $k-1 \leq uN < k$, $k = 1, \ldots, N$. Using the specific covariance structure of $K_n$, one can show by elementary but tedious calculations that the process
$$
  v \mapsto \left(D(k, v) - N^{-1} K_n(1,v)\right)_{1 \leq k \leq N}
$$
has the same distribution as
$$
  v \mapsto N^{-1/2} \left(B_k(v) - \overline{B}_{\cdot, N}(v)\right)_{1 \leq k \leq N},
$$
where $B_1, \ldots, B_{N}$ are independent standard Brownian bridges and $\bar{B}_{\cdot, N}$ denotes the arithmetic mean of $B_1, \ldots, B_{N}$; we do this in Lemma \ref{lem:increments_vector} below. Since
$$
\left\{\sqrt{n}\left[\mathcal{R}_1\left(\overline{\partial_1 C}\right) (u,v) - v\right]\right\}_{u,v \in [0,1]} = \left\{\mathcal{R}_1\left[\sqrt{n}\left(\overline{\partial_1 C}(u,v) - v\right)\right]\right\}_{u,v \in [0,1]},
$$
this observation implies that
$$
  \Big\{\sqrt{n}\left(\mathcal{R}_1\left(\overline{\partial_1 C}\right) (u,v) - v\right)\Big\}_{u,v \in [0,1]} \stackrel{\mathcal{D}}{=} \Big\{\sqrt{N}\left(F_{N,v}^{-1}(1-u) - \overline{B}_{\cdot, N}(v)\right)\Big\}_{u,v \in [0,1]},
$$
where $F_{N,v}$ is the empirical distribution function of $B_1(v), \ldots, B_{N}(v)$ and $F_{N,v}^{-1}$ denotes the generalised inverse of $F_{N,v}$. 
\begin{align}
\begin{split}
  \label{eq:quantilreduktion}
  &\Bigg\{\sqrt{n}\Big(\mathcal{R}_1\left(\overline{\partial_1 C}\right) (u,v) - v - \sqrt{\frac{N}{n}} \Phi_{v(1-v)}^{-1}(1-u)\Big)\Bigg\}_{u,v \in [0,1]} \\
  &\stackrel{\mathcal{D}}{=} \Bigg\{\sqrt{N}\left(F_{N,v}^{-1}(1-u) - \Phi_{v(1-v)}^{-1}(1-u)\right) - \sqrt{N} \, \overline{B}_{\cdot, N}(v)\Bigg\}_{u,v \in [0,1]}. 
\end{split}
\end{align}
Consequently,~\eqref{eq:pfouthelp1} follows if we prove that 
\begin{align}\label{eq:pfouthelp2}
\sqrt{N}\left(\int_0^t F_{N,v}^{-1}(1-u) - \Phi_{v(1-v)}^{-1}(1-u) - \overline{B}_{\cdot, N}(v)  ~\mathrm{d}u\right)_{(v,t) \in [0,1]^2} \rightsquigarrow \left(\int_0^t H(u,v) ~\mathrm{d}u\right)_{(v,t) \in [0,1]^2}.
\end{align}

Now, as an intermediate step, consider any i.i.d.\@ sequence $Y_1, \ldots, Y_n$ with a continuous distribution function $F$ and empirical distribution function $F_n$. Then $U_1, \ldots, U_n$ defined by $U_i = F(Y_i)$, $i = 1, \ldots, n$, are i.i.d. distributed uniformly on $[0,1]$. Let us denote the distribution function of $U_1, \ldots, U_n$ by $F_{U,n}$. We obtain
$$
  F_n^{-1}(s) - F^{-1}(s) = F^{-1}\left(F_{U,n}^{-1}(s)\right) - F^{-1}(s),
$$
and an application of Taylor's theorem with the Peano form of the remainder yields
$$
  F_n^{-1}(s) - F^{-1}(s) = \frac{1}{f\circ F^{-1}(s)}\left(F_{U,n}^{-1}(s) - s\right) + \left(\frac{1}{f\circ F^{-1}(s)} - \frac{1}{f\circ F^{-1}(s)}\right)\left(F_{U,n}^{-1}(s) - s\right)
$$
for some $s_n$ between $s$ and $F_{U,n}^{-1}(s)$. The function $f$ is the derivative of $F$, i.e.\@ the corresponding density function.

We can apply this to one part of the right-hand side of Eq.\@ \eqref{eq:quantilreduktion} and obtain
\begin{align}
    \label{eq:quantilreduktion_2}
     \sqrt{N}\left(F_{N,v}^{-1}(u) - \Phi_{v(1-v)}^{-1}(u)\right)                               = \frac{Q_{N}(u,v)}{\varphi\circ\Phi^{-1}(u)} + R_{n,3}(u,v),
\end{align}
where
\[
R_{n,3}(u,v) := \left(\frac{1}{\varphi\circ\Phi^{-1}(\xi_{N,u})} - \frac{1}{\varphi\circ\Phi^{-1}(u)}\right) Q_{N}(u,v),
\]
for some $\xi_{N,u}$ between $u$ and $\tilde{F}_{N,v}^{-1}(u)$, where
$$
  Q_{N}(u,v) = \sqrt{Nv(1-v)} \left\{\tilde{F}_{N,v}^{-1}(u) - u\right\},
$$
and $\tilde{F}_{N,v}$ is the empirical cumulative distribution function of $\Phi_{v(1-v)}(B_i(v))$, $i = 1, \ldots, N$. 

\medskip

Recall that we need to prove~\eqref{eq:pfouthelp2}. Let $a_N = N^{-\eta}$ for some $\eta > 1/2$ which is chosen sufficiently small so that $a_N$ satisfies the technical assumptions of Lemmas \ref{lem:aN_negl} and \ref{lem:uniform_quantiles}. By Lemma \ref{lem:diskretisierung_v} it suffices to consider the process
\[
\sqrt{N}\left(\int_{0}^{t} F_{N,g_{N}(v)}^{-1}(1-u) - \Phi_{g_{N}(v)(1-g_{N}(v))}^{-1}(1-u) - \overline{B}_{\cdot, N}(v)  ~\mathrm{d}u\right)_{(v,t) \in [0,1]^2}.
\]
The function $g_N$ is a certain discretisation of the unit interval, given by
\begin{align*}
g_N(u) := 
\left\{ 
\begin{array}{cc}
  N^{-2} \lceil N^2 u\rceil,  & u \in [0,1/2]  \\
  N^{-2} \lfloor N^2 u\rfloor,  & u \in (1/2,1]
\end{array}
\right.    
\end{align*}
Its value lies in the fact that it allows us to consider finite maxima over $j N^{-2}$, $j = 0, \ldots, N^2$, instead of uncountable suprema over $u, v \in [0,1]$. To apply the cited Lemma \ref{lem:diskretisierung_v}, note that $\int_{0}^t f(1-u) ~\mathrm{d}u = \int_{1-t}^1 f(u) ~\mathrm{d}u = \int_{0}^1f(u) ~\mathrm{d}u - \int_{0}^t f(u) ~\mathrm{d}u$ and apply the triangle inequality.

By Lemma~\ref{lem:aN_negl} we can further restrict our attention to
\[
\sqrt{N}\left(\int_{[a_N,1-a_N]\cap[0,t]}F_{N,g_{N}(v)}^{-1}(1-u) - \Phi_{g_{N}(v)(1-g_{N}(v))}^{-1}(1-u) - \overline{B}_{\cdot, N}(v)  ~\mathrm{d}u\right)_{(v,t) \in [0,1]^2}.
\]
after noting that $\int_0^{a} h(1-u) ~\mathrm{d}u = \int_{1-a}^1 h(u) ~\mathrm{d}u$ and $\int_{1-a}^{1} h(1-u) du = \int_{a}^1 h(u) ~\mathrm{d}u$. Cutting off the domain of integration near the boundaries $0$ and $1$ allows us to ignore the `almost degenerate' behaviour of the integrand processes there. By Lemma \ref{lem:uniform_quantiles},
  \begin{align*}
     & \int_{a_N}^t \sqrt{N}\left(F_{N,g_N(v)}^{-1}(1-u) - \Phi_{g_N(v)(1-g_N(v))}^{-1}(1-u) - \bar{B}_{\cdot, N}(v)\right) ~\mathrm{d}u 
     \\
     & \quad = \int_{a_N}^t \frac{Q_N'(1-u, g_N(v))}{\varphi \circ \Phi^{-1}(1-u)} - \sqrt{N}\bar{B}_{\cdot, N}(v) ~\mathrm{d}u + o_\mathbb{P}(1)
  \end{align*}
  uniformly in $a_N \leq t \leq 1-a_N$ and $0 \leq v \leq 1$, where
  \begin{align*}
  Q_N'(u,v) & = \sqrt{Nv(1-v)} \left\{\tilde{F}_{N,v}(u) - u\right\}, \\
  \tilde{F}_{N,v}(u) &= \frac{1}{N} \sum_{i=1}^N \textbf{1}\left(\Phi_{v(1-v)}(B_i(v)) \leq u\right).
  \end{align*}
This is a Bahadur-Kiefer type approximation, with the important feature that we are allowing the argument $1-u \in [a_N, 1-a_N]$ to get arbitrarily close to $0$ and $1$ for growing $N$, since $a_N \to 0$. Furthermore,
  \begin{align}
    \begin{split}
      \label{eq:gesamtintegral}
      & \int_{a_N}^t \frac{Q_N'(1-u, g_N(v))}{\varphi \circ \Phi^{-1}(1-u)} - \sqrt{N}\bar{B}_{\cdot, N}(v) ~\mathrm{d}u                                                 \\
      & \quad = \int_{a_N}^t \frac{\{u(1-u)\}^\alpha}{\varphi \circ \Phi^{-1}(1-u)} \{u(1-u)\}^{-\alpha} Q_N'(1-u, g_N(v)) - \sqrt{N}\bar{B}_{\cdot, N}(v) ~\mathrm{d}u
    \end{split}
  \end{align}
  for any $\alpha > 0$. Set $\alpha = 1/32$\footnote{Any value $\alpha \in (0,1/32]$ will work for the subsequent arguments.}. By Lemma \ref{lem:konvergenz_QN_prime}, 
  the process $(u,v) \mapsto \{u(1-u)\}^{-\alpha} Q_N'(u, g_N(v))$ converges in distribution to some limiting process $L$. The function
  $$
    u \mapsto \frac{\{u(1-u)\}^{\alpha}}{\varphi \circ \Phi^{-1}(1-u)}
  $$
  is integrable on the unit interval. At this point, we can return to the integral from $0$ to $t$; since $\sup_{u,v \in [0,1]}\{u(1-u)\}^{-\alpha} \big|Q_N'(1-u, g_N(v))\big|$ is bounded in probability by Lemma \ref{lem:konvergenz_QN_prime}. In other words, the right-hand side in Eq.\@ \eqref{eq:gesamtintegral} is equal to
  $$
    \int_{0}^t \frac{\{u(1-u)\}^{\alpha}}{\varphi \circ \Phi^{-1}(1-u)} \{u(1-u)\}^{-\alpha} Q_N'(1-u, g_N(v)) - \sqrt{N}\bar{B}_{\cdot, N}(v) ~\mathrm{d}u + o_\mathbb{P}(1),
  $$
  where the remainder is uniform in $v,t \in [0,1]$. The processes $(u,v) \mapsto \{u(1-u)\}^{-\alpha} Q_N'(1-u,v)$ and $v \mapsto \sqrt{N} \bar{B}_{\cdot, N}(v)$ converge not only individually but also jointly: Since they are certainly jointly asymptotically tight \citep[cf.\@ Lemma 1.4.3 in][]{vandervaart_wellner:weak_convergence} it suffices to show the finite dimensional convergence. But this is easily done by way of a multivariate central limit theorem, noting that the covariance function of $Q_N'$ and $\sqrt{N} \bar{B}_{\cdot,N}$ is equal to that in Eq.\@ \eqref{eq:kovarianzfunktion_2} for all $N$. Finally, the map 
  \begin{align*}
    \ell^\infty([0,1]^2) \times \ell^\infty([0,1]) & \to \ell^\infty([0,1])
    \\
    (f,g)                        & \mapsto \left[(t,v) \mapsto \int_0^t \frac{\{u(1-u)\}^{\alpha}}{\varphi \circ \Phi^{-1}(1-u)} f(1-u,v) - g(v) ~\mathrm{d}u\right]
  \end{align*}
  is continuous, and so the weak convergence in \eqref{eq:pfouthelp2} follows with the continuous mapping theorem. This completes the proof of Theorem~\ref{thm:asymptotik}. \hfill $\Box$

\subsection{Proof of Theorem \ref{thm:checkerboard}}
Suppose that $u \in [j/N_1, (j+1)/N_1)$ and $v \in [k/N_2, (k+1)/N_2)$ and write $\lambda = N_2 v - k$. Then
\begin{align}
  \begin{split}
    \label{eq:checkerboard_bound}
    & \partial_1 C_n^\#(u,v)                                                                                                                                                                                                                                                        \\
    & = N_1 \left\{C_n^\#\left(\frac{j+1}{N_1}, v\right) - C_n^\#\left(\frac{j}{N_1}, v\right)\right\}                                                                                                                                                                       \\
    & = N_1 \left\{\lambda C_n\left(\frac{j+1}{N_1}, \frac{k+1}{N_2}\right) + (1-\lambda)C_n\left(\frac{j+1}{N_1}, \frac{k}{N_2}\right) - \lambda C_n\left(\frac{j}{N_1}, \frac{k+1}{N_2}\right) + (1-\lambda) C_n\left(\frac{j}{N_1}, \frac{k}{N_2}\right)\right\}           \\
    & = N_1\left[ \lambda \left\{ C_n\left(\frac{j+1}{N_1}, \frac{k+1}{N_2}\right) - C_n\left(\frac{j}{N_1}, \frac{k+1}{N_2}\right) \right\} + (1-\lambda) \left\{C_n\left(\frac{j+1}{N_1}, \frac{k}{N_2}\right) -C_n\left(\frac{j}{N_1}, \frac{k}{N_2}\right)\right\}\right] \\
    & = \lambda \widehat{\partial_1 C}\left(u, \frac{k+1}{N_2}\right) + (1-\lambda)\widehat{\partial_1 C}\left(u, \frac{k}{N_2}\right).
  \end{split}
\end{align}
By Theorem 1 in \cite{AnevskiFougeres2019}, the monotone rearrangement is a contraction in the supremum norm on the set of all functions which are bounded on the unit interval. Therefore,
\begin{align}
  \begin{split}
    \label{eq:checkerboard_bound_2}
    \sup_{0 \leq t,v \leq 1}\left|\int_0^t \mathcal{R}_1\left(\partial_1 C_n^\#\right)(s,v) - \mathcal{R}_1\left(\widehat{\partial_1 C}\right) (s,v) ~\mathrm{d}s\right| & \leq \sup_{0 \leq u,v \leq 1}\left|\mathcal{R}_1\left(\partial_1 C_n^\#\right)(u,v) - \mathcal{R}_1\left(\widehat{\partial_1 C}\right) (u,v)\right| \\
    & \leq \sup_{0 \leq u,v \leq 1}\left|\partial_1 C_n^\#(u,v) - \widehat{\partial_1 C} (u,v)\right|                               \\
    & \leq \sup_{0 \leq u \leq 1}\sup_{(v, v') \in V_n}\left|\widehat{\partial_1 C} (u,v) - \widehat{\partial_1 C} (u,v')\right|,
  \end{split}
\end{align}
where the inner supremum in the last line is taken over $V_n := \{(v,v'): 0 \leq v,v' \leq 1, |v-v'| \leq N_2^{-1}\}$. This last inequality is a consequence of Eq.\@ \eqref{eq:checkerboard_bound}.

We now use the fact that for any $0 \leq u,u',v,v' \leq 1$, it holds that
\begin{equation}
    \label{eq:almost_lipschitz_empirical_copula}
  |C_n(u,v) - C_n(u',v')| \leq |u-u'| + |v-v'| + \frac{2}{n}
\end{equation}
almost surely. This can be seen by the following argument: For fixed $v$, the function $u \mapsto C_n(u,v) = C_n(\lfloor nu \rfloor/n,v)$ can only jump at $u = j/n$ by at most $1/n$. Since the number of possible jumps between $\lfloor nu \rfloor/n$ and $\lfloor nu' \rfloor/n$ is at most equal to $|\lfloor nu \rfloor - \lfloor nu' \rfloor|$, we have 
$$
\left|C_n(u,v) - C_n(u',v)\right| = |C_n(\lfloor nu \rfloor/n,v) - C_n(\lfloor nu' \rfloor/n,v)| \leq \frac{|\lfloor nu \rfloor - \lfloor nu' \rfloor|}{n} \leq |u-u'| + \frac{1}{n}.
$$
The same argument works for $v \mapsto C_n(u,v)$ for fixed $u$, thus leading to Eq.\@ \eqref{eq:almost_lipschitz_empirical_copula}.
Now, using the definition of $\widehat{\partial_1 C}$,
$$
\widehat{\partial_1 C}(u,v) = N_1\left[C_n\left(\frac{j+1}{N_1}, v\right) - C_n\left(\frac{j}{N_1}, v\right)\right]
$$
for $j/N_1 \leq u < (j+1)/N_1$, we see that the last line in Eq.\@ \eqref{eq:checkerboard_bound_2} is almost surely bounded by $2N_1/N_2 + 4N_1/n$, which is $o(1/\sqrt{n})$ by assumption. This concludes the proof of Theorem \ref{thm:checkerboard}. \hfill $\Box$

\section{Funding}
The research of Stanislav Volgushev was partially supported by a Discovery Grants from the Natural Sciences and Engineering Research Council of Canada (RGPIN-2024-05528). The work  of Holger Dette and Marius Kroll has been partially supported
by the Deutsche Forschungsgemeinschaft (DFG):
TRR 391 {\it Spatio-temporal Statistics for the Transition of Energy and Transport} (520388526);
Research unit 5381 \textit{Mathematical Statistics in the Information Age} (460867398).

\bibliographystyle{abbrvnat}
\bibliography{rearrangements}

\newpage

\begin{appendix}

\section{Appendix: details on intermediate technical steps and their proofs.} \label{sec:proof-cop}
Recall that for a CDF $F$ we define the generalized inverse by 
\[
F^{-1}(p) := \inf\{x : F(x) \ge p\}, \quad p \in (0,1).
\]
In this section, we extend this definition by setting $F^{-1}(0)$ as the lower bound for the support of $F$ and $F^{-1}(1)$ as the upper bound of the support of $F$. With this definition $F^{-1}(0)$ can take the value $-\infty$ and $F^{-1}(1)$ the value $+\infty$.

Throughout this section, we continue to use the notation introduced in Section \ref{sec:proof_outline}. For ease of notation, we write $N$ instead of $N_1$. We define the function $\psi_2(x) = \exp(x^2) - 1$ and write
$$
  \|X\|_{\psi_2} = \inf\left\{ C > 0 ~\big|~ \mathbb{E}\left[\psi_2\left(\frac{|X|}{C}\right)\right] \leq 1\right\}
$$
for the corresponding Orlicz norm.

\begin{lemma}
  \label{lem:increments_vector}
  The vector-valued process $v \mapsto \left(D_C(k, v) - N^{-1} B_C(1,v)\right)_{1 \leq k \leq N}$ has the same distribution as $v \mapsto N^{-1/2} \left(B_k(v) - \bar{B}_{\cdot, N}(v)\right)_{1 \leq k \leq N}$, where $B_1, \ldots, B_N$ are independent standard Brownian bridges and $\bar{B}_{\cdot, N}$ denotes the arithmetic mean of $B_1, \ldots, B_N$.
\end{lemma}
\begin{proof}[Proof of Lemma~\ref{lem:increments_vector}]
  The covariance matrix of $N^{-1/2}(B_k(s) - \bar{B}_{\cdot, N}(s))_{1 \leq k \leq N}$ and $N^{-1/2} (B_k(t) - \bar{B}_{\cdot, N}(t))_{1 \leq k \leq N}$ is
  $$
    \Sigma(s,t) =  \frac{(s \land t - st)}{N}\begin{pmatrix}
      \frac{N-1}{N} & \cdots & - \frac{1}{N} \\ \vdots & \ddots & \vdots \\ -\frac{1}{N} & \cdots & \frac{N-1}{N}
    \end{pmatrix}.
  $$
  The claim follows by elementary calculations, using the fact that $B_C$ is a $C$-Brownian bridge, i.e.\@ it has covariance function
  \begin{align*}
  &\mathrm{Cov}[B_C(u,v), B_C(u',v')] \\
  &= \mathbb{P}_C\left(\textbf{1}_{(-\infty,u]\times (-\infty,v]}\textbf{1}_{(-\infty,u']\times (-\infty,v']}\right) - \mathbb{P}_C \textbf{1}_{(-\infty,u]\times (-\infty,v]} \mathbb{P}_C \textbf{1}_{(-\infty,u']\times (-\infty,v']} \\
  &= C(u\land u', v \land v') - C(u,v)C(u',v'),
  \end{align*}
  and $C(u,v) = uv$.
\end{proof}

\begin{lemma}
  \label{lem:brownian_increments_orlicz}
Let $W$ be a standard Brownian motion and $B$ a standard Brownian bridge on the unit interval. Then there is a uniform constant $C_0$ such that for any $0 < \delta \leq 1$ and any $p > 0$,
  $$
    \left\|\sup_{|s-t| \leq \delta} |W(s) - W(t)|\right\|_{\psi_2} + \left\|\sup_{|s-t| \leq \delta} |B(s) - B(t)|\right\|_{\psi_2}  \leq C_0 \left\{\int_0^{\delta^p} \sqrt{- \log x} ~\mathrm{d}x + \sqrt{- 2p \delta \log \delta}\right\}.
  $$
  For $p = 1/2$, the second summand can be absorbed into the integral to yield the bound
  $$
    C_0 \int_0^{\sqrt{\delta}} \sqrt{- \log x} ~\mathrm{d}x,
  $$
which in turn can be bounded by $c_r \delta^{r/2}$ for any $r < 1$, where $c_r$ is a constant depending only on $r$.
\end{lemma}
\begin{proof}[Proof of Lemma~\ref{lem:brownian_increments_orlicz}] 
Throughout the proof, $C$ is a positive constant that may change from line to line. Any Gaussian process $X$ is sub-Gaussian with respect to its own standard deviation semimetric $\rho(s,t) = \sqrt{\mathrm{Var}(X(s) - X(t))}$ and satisfies $\|X(s) - X(t)\|_{\psi_2} \leq \sqrt{6} \rho(s,t)$ \citep[p.\@ 101 in][]{vandervaart_wellner:weak_convergence}. The standard deviation metric for a Brownian motion is $\rho_{\mathrm{BM}}(s,t) = \sqrt{\mathrm{Var}(W(s) - W(t))} = \sqrt{|s-t|}$, and the covering number $N(x, \rho_{\mathrm{BM}})$ of the unit interval with respect to this metric is $ \lesssim x^{-2}$. By Theorem 2.2.4 in \cite{vandervaart_wellner:weak_convergence} {and the paragraph just before that theorem}, there exists some constant $K$ such that 
  \begin{align*}
    \left\|\sup_{\rho_{\mathrm{BM}}(s,t) \leq \sqrt{\delta}} |W(s) - W(t)|\right\|_{\psi_2} & \leq K \left\{\int_0^{\eta} \sqrt{\log(N(x/2, \rho_{\mathrm{BM}})+1)} ~\mathrm{d}x + \sqrt{\delta} \sqrt{2 \log( N(\eta/2, \rho_{\mathrm{BM}})+1)}\right\} 
    \\
 & \lesssim \left\{\int_0^{\eta} \sqrt{- \log x} ~\mathrm{d}x + \sqrt{-2 \delta \log \eta}\right\}
  \end{align*}
  for any $\eta \in (0,1]$. In the second line we have used the fact that $\log(x+1) \leq 2 \log x$ for all $x \geq 2$ and $N(\eta/2,\rho_{\mathrm{BM}}) \geq 2$. If we choose $\eta = \delta^p$, we get the result for the Brownian motion, since $\rho_{BM}(s,t) = \sqrt{|s-t|} \leq \sqrt{\delta}$ if and only if $|s-t| \leq \delta$. The proof for the Brownian bridge is mostly the same, noting that in this case the standard deviation metric is $\rho_{\mathrm{BB}}(s,t) = \sqrt{|s-t|(1-|s-t|)} \leq \rho_{BM}(s,t)$. This implies that $N(x, \rho_{\mathrm{BB}}) \leq N(x, \rho_{\mathrm{BM}})$, and from here we can proceed as before.

 To prove the bound in terms of $\delta^{r/2}$, just note that
$$
\int_0^{\sqrt{\delta}} \sqrt{\log \frac{1}{x}} ~\mathrm{d}x \lesssim \int_0^{\sqrt{\delta}} x^{-s/2} ~\mathrm{d}x \lesssim \delta^{1/2 - s/4}
$$
for all $0 < s < 2$, where $\lesssim$ is hiding constants depending only on $s$. By choosing $s$ sufficiently small, this last term can be made equal to $\delta^{r/2}$ for any $r < 1$.
\end{proof}

\begin{lemma}
  \label{lem:aN_negl} Recall the definition of $F_{N,v}$ above equation \eqref{eq:quantilreduktion} and assume $N \to \infty$. If $a_N = o\left(N^{-1/2} (\log N)^{-1}\right)$, then
  \begin{align*}
     & \sup_{v \in [0,1]} \int_0^{a_N} \sqrt{N}\left(\big|F_{N,v}^{-1}(u)\big| + \big|\overline{B}_{\cdot, N}(v)\big| + \big|\Phi_{v(1-v)}^{-1}(u)\big|\right) ~\mathrm{d}u           \\
     & \quad + \sup_{v \in [0,1]} \int_{1-a_N}^1 \sqrt{N}\left(\big|F_{N,v}^{-1}(u)\big| + \big|\overline{B}_{\cdot, N}(v)\big| + \big|\Phi_{v(1-v)}^{-1}(u)\big|\right) ~\mathrm{d}u  = o_\P(1).
  \end{align*} 
\end{lemma}
\begin{proof}[Proof of Lemma~\ref{lem:aN_negl}]
  By Lemma \ref{lem:brownian_increments_orlicz}, there is a constant $K > 0$ such that
  $$
    \left\|\sup_{0 \leq s \leq 1} |B(s)| \right\|_{\psi_2} \leq \left\|\sup_{0 \leq s,t \leq 1} |B(s) - B(t)| \right\|_{\psi_2} \leq K \int_0^1 \sqrt{-\log x} ~\mathrm{d}x < \infty.
  $$
  If $B_1, \ldots, B_N$ are $N$ standard Brownian bridges, this implies
  $$
    \left\|\max_{i = 1, \ldots , N} \sup_{0 \leq v \leq 1} |B_i(v)| \right\|_{\psi_2} \leq K_1 \sqrt{\log (N + 1)}
  $$
  for another constant $K_1 > 0$ by Lemma 2.2.2 in \cite{vandervaart_wellner:weak_convergence}, and so
  $$
    \sup_{0 \leq u,v \leq 1} \left|F_{N,v}^{-1}(u)\right| \leq \max_{i = 1, \ldots , N} \sup_{0 \leq v \leq 1} |B_i(v)|  = \mathcal{O}_\mathbb{P}\left(\sqrt{\log N}\right).
  $$
  Because $a_N = o\left((N \log N)^{-1/2}\right)$ this also means that
  $$
    \sup_{0 \leq v \leq 1} \left\{\int_0^{a_N} \left|F_{N,v}^{-1}(u)\right| ~\mathrm{d}u + \int_{1 - a_N}^1 \left|F_{N,v}^{-1}(u)\right| ~\mathrm{d}u \right\}= o_\mathbb{P}\left(N^{-1/2}\right)
  $$
  as well as
  $$
    \sup_{0 \leq v \leq 1} \left\{\int_0^{a_N} \left|\bar{B}_{\cdot, N}(u)\right| ~\mathrm{d}u + \int_{1-a_N}^1 \left|\bar{B}_{\cdot, N}(u)\right| ~\mathrm{d}u\right\} = o_\mathbb{P}\left(N^{-1/2}\right).
  $$
  By identity [3.9.1] in \cite{patel_read:handbook_normal}, we have $\Phi^{-1}(1-x) \leq \sqrt{-2 \log x} + e(x)$ for $0 < x \leq 1/2$, where $e(x)$ is an error function with $|e(x)| < 0.003$ for all $x \leq 1/2$. Since $\Phi^{-1}_{v(1-v)}(u) = \sqrt{v(1-v)}\Phi^{-1}(u) = -\sqrt{v(1-v)} \Phi^{-1}(1-u)$ for all $v \in (0,1)$, we obtain for sufficiently large $N$
  \begin{align*}
    \sup_{0 \leq v \leq 1} \int_0^{a_N} \left|\Phi_{v(1-v)}^{-1}(u)\right| ~\mathrm{d}u & \leq c \int_0^{a_N} \sqrt{- \log u} + |e(u)| ~\mathrm{d}u \\
     & \leq c \int_0^{a_N} - \log u ~\mathrm{d}u + \mathcal{O}(a_N) 
  \end{align*}
  for some constant $c > 0$. Writing $b_N = a_N \lor N^{-1}$, we can further bound the integral by 
  $$
  \int_{0}^{a_N} -\log u ~\mathrm{d}u \leq \int_{0}^{b_N} -\log u ~\mathrm{d}u = b_N + b_N \log\frac{1}{b_N} \leq b_N + b_N\log N = o\left(N^{-1/2}\right).
  $$
  Therefore,
  $$
  \sup_{0 \leq v \leq 1} \int_0^{a_N} \left|\Phi_{v(1-v)}^{-1}(u)\right| ~\mathrm{d}u = \mathcal{O}(a_N) + o\left(N^{-1/2}\right) = o\left(N^{-1/2}\right).
  $$
  A similar argument works for the upper tail, and so
  $$
    \int_0^{a_N} \sqrt{N}\left(F_{N,v}^{-1}(u) - \bar{B}_{\cdot, N}(v) - \Phi_{v(1-v)}^{-1}(u)\right) ~\mathrm{d}u = o_\mathbb{P}(1).
  $$
  and
  $$
    \int_{1-a_N}^1 \sqrt{N}\left(F_{N,v}^{-1}(u) - \bar{B}_{\cdot, N}(v) - \Phi_{v(1-v)}^{-1}(u)\right) ~\mathrm{d}u = o_\mathbb{P}(1).
  $$
\end{proof}

We now introduce a discretisation of the arguments of our processes. For each $N$, let
\begin{align*}
g_N(u) := 
\left\{ 
\begin{array}{cc}
  N^{-2} \lceil N^2 u\rceil,  & u \in [0,1/2]  \\
  N^{-2} \lfloor N^2 u\rfloor,  & u \in (1/2,1]
\end{array}
\right.    
\end{align*}

In what follows, we will without loss of generality assume that $N$ is even. If $N$ is odd, replace $N^2$ in the definition of $g_N$ by $4N^2$ below and the arguments go through in the same way.

The function $g_N$ maps the unit interval to the points $j N^{-2}$, $j = 0, \ldots, N^2$. Although it is purely notation, we will denote these points by $u_j$ or $v_j$, depending on whether we are considering $g_N(u)$ or $g_N(v)$. Thus, we may write e.g.\@ `for all $v_{j-1} \leq v \leq v_j$ and $u_{j-1} \leq u \leq u_j$', when in fact $u_j = v_j = j N^{-2}$ for all $j = 0, \ldots, N^2$.
\begin{lemma}
  \label{lem:diskretisierung_v}
Assume $N \to \infty.$ The random variable
  \[
      \sup_{0 \leq t \leq 1} \sup_{v \in [0,1]}\sqrt{N}\left|\int_0^t F_{N,v}^{-1}(u) - \Phi_{v(1-v)}^{-1}(u) ~\mathrm{d}u - \int_0^t F_{N,g_N(v)}^{-1}(u) - \Phi_{g_N(v)(1-g_N(v))}^{-1}(u) ~\mathrm{d}u\right|
  \]
  converges to $0$ in probability.
\end{lemma}
\begin{proof}[Proof of Lemma~\ref{lem:diskretisierung_v}]
If $v \in \{0,1\}$, we have $g_N(v) = v$ and so it suffices to consider the supremum over $v \in (0,1)$. In what follows, we consider the case $v \in (0,1/2]$, so that $v \in (v_{j-1},v_j]$ implies $g_N(v) = v_j$. The case $v \in (1/2,1)$ can be treated similarly.

For $v \in (0,1/2]$, $\Phi_{v(1-v)}^{-1} = \sqrt{v(1-v)}\,\Phi^{-1}$ and so,
  \begin{equation}
    \label{eq:hölder_eq_bound}
    \sqrt{N}\left|\int_0^t \Phi_{v(1-v)}^{-1}(u) - \Phi_{v_j(1-v_j)}^{-1}(u) ~\mathrm{d}u\right| = \sqrt{N}\left|\left\{\sqrt{v(1-v)} - \sqrt{v_j(1-v_j)}\right\} \int_0^t \Phi^{-1}(u) ~\mathrm{d}u\right|.
  \end{equation}
  Now $\int_0^t \Phi^{-1}(u) ~\mathrm{d}u$ is bounded absolutely by some constant uniformly with respect to  $t$. Note that the derivative of $x \mapsto \sqrt{x(1-x)}$ is decreasing on $(0,1/2)$ and the function is itself increasing, so for $v_{j-1} < v \le v_j \leq 1/2$
  \[
  \left|\sqrt{v(1-v)} - \sqrt{v_j(1-v_j)}\right| \leq \left|\sqrt{v_{j-1}(1-v_{j-1})} - \sqrt{v_j(1-v_j)}\right| \leq \sqrt{v_1(1-v_1)}\leq 1/N.
  \]
  Since
  $$
  \left|\int_0^t \Phi^{-1}(u) ~\mathrm{d}u\right| \leq \int_0^t |\Phi^{-1}(u)| ~\mathrm{d}u \leq \int_0^1 |\Phi^{-1}(u)| ~\mathrm{d}u = \mathbb{E} |\Phi^{-1}(U)| = \sqrt{\frac{2}{\pi}},
  $$
  for all $0 \leq t \leq 1$, we see from Eq.\@ \eqref{eq:hölder_eq_bound} that
  \begin{equation}
    \label{eq:diskretisierung_v_1}
    \sup_{0 \leq t \leq 1} \max_{j = 1, \ldots, N^2/2} \sup_{v_{j-1} < v \leq v_j} \sqrt{N} \left|\int_0^t \Phi_{v(1-v)}^{-1}(u) - \Phi_{v_j(1-v_j)}^{-1}(u) ~\mathrm{d}u\right| = \mathcal{O}\left(\frac{1}{\sqrt{N}}\right).
  \end{equation}
  It remains to consider the difference
  $$
    \sup_{0 \leq t \leq 1} \max_{j = 1, \ldots, N^2/2} \sup_{v_{j-1} < v \leq v_j} \left|\int_0^t F_{N,v}^{-1}(u) - F_{N,v_j}^{-1}(u) ~\mathrm{d}u\right|.
  $$
  For any fixed $j$, we can bound the difference of the quantile functions by
  $$
    \sup_{0 \leq u \leq 1} \left| F_{N,v}^{-1}(u) - F_{N,v_j}^{-1}(u) \right| \leq \max_{i = 1, \ldots, N} |B_i(v) - B_i(v_j)|.
  $$
  The most direct way to see this is by using the contraction property of the monotone rearrangement operator with respect to the supremum norm; see Theorem 1 in \cite{AnevskiFougeres2019}. Therefore,
  $$
    \max_{j = 1, \ldots, N^2/2} \sup_{v_{j-1} < v \leq v_j}\sup_{0 \leq u \leq 1} \left| F_{N,v}^{-1}(u) - F_{N,v_j}^{-1}(u) \right| \leq  \max_{i = 1, \ldots, N} \max_{j = 1, \ldots, N^2/2} \sup_{v_{j-1} < v \leq v_j} |B_i(v) - B_i(v_j)|.
  $$
By Lemma \ref{lem:brownian_increments_orlicz}, we have
  $$
    \left\|\sup_{v_{j-1} < v \leq v_j} |B_i(v) - B_i(v_j)|\right\|_{\psi_2} \leq \left\|\sup_{|v-v'| \leq N^{-2}} |B_i(v) - B_i(v')|\right\|_{\psi_2} \leq c_\gamma N^{-\gamma}
  $$
  for any $\gamma < 1$ and constants $c_\gamma > 0$. Therefore, by applying Lemma 2.2.2 in \cite{vandervaart_wellner:weak_convergence} to the maxima over $i$ and $j$ simultaneously, we obtain
  $$
    \left\|\max_{i = 1, \ldots, N} \max_{j = 1, \ldots, N^2/2} \sup_{v_{j-1} < v \leq v_j} |B_i(v) - B_i(v_j)|\right\|_{\psi_2} \leq K \frac{\sqrt{ \log N}}{N^{\gamma}}
  $$
  for some constant $K = K(\gamma)> 0$, and this upper bound is $o(N^{-1/2})$ if $\gamma > 1/2$. Thus,
  \begin{align}
    \begin{split}
      \label{eq:QNprime_v_diskret}
      & \sup_{0 \leq t \leq 1} \max_{j = 1, \ldots, N^2/2} \sup_{v_{j-1} < v \leq v_j} \left|\int_0^t F_{N,v}^{-1}(u) - F_{N,v_j}^{-1}(u) ~\mathrm{d}u\right| \\
      & \quad \leq \max_{i = 1, \ldots, N} \max_{j = 1, \ldots, N^2/2} \sup_{v_{j-1} < v \leq v_j} |B_i(v) - B_i(v_j)|                                      \\
      & \quad = o_\mathbb{P}\left(\frac{1}{\sqrt{N}}\right).
    \end{split}
  \end{align}
  Together with similar arguments for $v > 1/2$ and Eq.\@ \eqref{eq:diskretisierung_v_1} proves our claim.
\end{proof}

\begin{lemma}
  \label{lem:gaussian_lemma}
  Let $(Z,Z')$ be jointly normal with be standard normal margins and arbitrary correlation and write $\Delta = Z' - Z$, $\sigma^2 = \mathrm{Var}(\Delta)$. If $0 < \sigma^2 \leq 2$, then there is a universal constant $c$ such that
  $$
    \mathbb{P}(Z \in [s - |\Delta|, s + |\Delta|]) \leq c \,\frac{\sigma \varphi(s/5)}{\sqrt{1 - \sigma^2/4}}
  $$
  for all $s \in \mathbb{R}$ where $\varphi$ denotes the standard normal density.
\end{lemma}
\begin{proof}[Proof of Lemma~\ref{lem:gaussian_lemma}]
  We can show by elementary calculations that $\rho = \mathrm{Cor}(Z,\Delta) = -\sigma/2$. Since $(Z,\Delta)$ follows a multivariate normal distribution, the conditional distribution of $Z$ given $\Delta$ is $\mathcal{N}\left(\rho \Delta/\sigma, 1 - \rho^2\right)$. By the law of total probability,
  \begin{align*}
    \mathbb{P}(Z \in [s - |\Delta|, s + |\Delta|]) & = \mathbb{E}\left[\mathbb{P}(Z \in [s - |\Delta|, s + |\Delta|] ~|~ \Delta)\right]                         \\
  & = \mathbb{E}\left[\Phi\left(\frac{s + |\Delta| - \rho \Delta/\sigma}{\sqrt{1 - \rho^2}}\right) - \Phi\left(\frac{s - |\Delta| - \rho \Delta/\sigma}{\sqrt{1 - \rho^2}}\right)\right]                             \\
   & = \mathbb{E}\left[\Phi\left(\frac{s + |\Delta| + \Delta/2}{\sqrt{1 - \rho^2}}\right) - \Phi\left(\frac{s - |\Delta| + \Delta/2}{\sqrt{1 - \rho^2}}\right)\right]                                             \\
   & = \int_{-\infty}^\infty \left\{\Phi\left(\frac{s + \sigma|u| + \sigma u/2}{\sqrt{1 - \rho^2}}\right) - \Phi\left(\frac{s - \sigma|u| + \sigma u/2}{\sqrt{1 - \rho^2}}\right)\right\} \varphi(u) ~\mathrm{du}
  \end{align*}
  For any $a \leq b$ and $0 < r < 1$, we can bound $\Phi(b/r) - \Phi(a/r)$ by $r^{-1}(b-a)\varphi(b/r) \leq r^{-1}(b-a) \varphi(b)$ if $b \leq 0$ and by $r^{-1}(b-a) \varphi(a/r) \leq r^{-1}(b-a)\varphi(a)$ if $a \geq 0$. In all other cases, we can trivially bound it by $r^{-1}(b-a)$. If $|u| \leq |s|/5$, then $3\sigma |u|/2 \leq 3|s|/(5\sqrt{2}) \leq |s|/2$. This also implies that the interval with endpoints $s \pm 3\sigma |u|/2$ is contained either in $[s/2, \infty)$ or $(-\infty, s/2]$, depending on the sign of $s$. The final integral in the above equation can therefore be bounded by
  \begin{align*}
     & \frac{1}{\sqrt{1-\rho^2}} \left\{\int_{-\infty}^{-|s|/5} 2|u|\sigma\varphi(u) ~\mathrm{d}u + \int_{-|s|/5}^{|s|/5} 2 |u| \sigma \varphi(u) \varphi\left(\frac{s}{2}\right) ~\mathrm{d}u + \int_{|s|/5}^\infty 2|u|\sigma \varphi(u) ~\mathrm{d}u\right\} \\
     & \quad \leq \mathrm{const} \cdot \frac{\sigma}{\sqrt{1-\rho^2}}\left\{\varphi\left(\frac{s}{2}\right) + \mathrm{exp}\left(-\frac{(s/5)^2}{2}\right)\right\}                                                                                               \\
     & \quad \leq \mathrm{const} \cdot \frac{\sigma}{\sqrt{1 - \rho^2}} \varphi\left(\frac{s}{5}\right)
  \end{align*}
  for some universal constants. This proves our claim since $\rho = -\sigma/2$.
\end{proof}

Define the processes $Q_N$ and $Q_N'$ on $[0,1]^2$ by
\begin{equation}
\begin{split}
  Q_N(u,v)  & = \sqrt{Nv(1-v)} \left\{\tilde{F}_{N,v}^{-1}(u) - u\right\}, \\
  Q_N'(u,v) & = \sqrt{Nv(1-v)} \left\{\tilde{F}_{N,v}(u) - u\right\},
\end{split}  \label{holger2}   
\end{equation}
where
$$
  \tilde{F}_{N,v}(u) = \frac{1}{N} \sum_{i=1}^N \textbf{1}\left(\Phi_{v(1-v)}(B_i(v)) \leq u\right).
$$

\begin{lemma}
  \label{lem:moment_lemma_FN_tilde}
  For any $k \in \mathbb{N}$ and  $\varepsilon > 0$, there exist two constants $c_k$ and $c'_\varepsilon$ such that for all $0 < \alpha < \min\{1/(2k), 1/52\}$ it holds that
  \begin{align*}
     & \left\|\{u(1-u)\}^{-\alpha} \sqrt{N}\left\{\tilde{F}_{N,v}(u) - u\right\} - \{u'(1-u')\}^{-\alpha} \sqrt{N}\left\{\tilde{F}_{N,v'}(u') - u'\right\}\right\|_{L_{2k}} \\
     & \quad \leq c_k \max\left\{N^{-1/2 + \alpha}, |u-u'|^{1/2 - \alpha} + c'_\varepsilon |v-v'|^{1/4}\right\}
  \end{align*}
  for all $N \in \mathbb{N}$, $0 \leq u, u' \leq 1$ and $\varepsilon < v,v' < 1-\varepsilon$. Here, we note that for $u>0$ sufficiently small, $\{u(1-u)\}^{-\alpha} \sqrt{N}\left\{\tilde{F}_{N,v}(u) - u\right\}= - u\{u(1-u)\}^{-\alpha}$ and for $u<1$ sufficiently close to $1$, $\{u(1-u)\}^{-\alpha} \sqrt{N}\left\{\tilde{F}_{N,v}(u) - u\right\}= (1- u)\{u(1-u)\}^{-\alpha}$. In this way, all expressions above are well defined for $u,u' = 0,1$.
\end{lemma}
\begin{proof}[Proof of Lemma~\ref{lem:moment_lemma_FN_tilde}]
Recall Rosenthal's inequality \citep[][Theorem 3]{rosenthal:1970}: for $p > 2$ there exists a constant $K_p$ depending only on $p$ such that for any collection of centered and independent random variables $D_1, \ldots, D_N$,
\begin{align}
    \label{holger1}
\mathbb{E}\left[\left|\sum_{i=1}^N D_i\right|^p\right] \leq K_p \max\left\{\sum_{i=1}^N \mathbb{E}\left[|D_i|^p\right], \left(\sum_{i=1}^N \mathbb{E}\left[D_i^2\right]\right)^{p/2}\right\}.
\end{align}
If additionally all $D_i$ have the same moments and $|D_i| \leq 1$ almost surely, then this yields
\begin{equation}
\label{eq:2k_moments_sum_bound}
\mathbb{E}\left[\left(N^{-1/2} \sum_{i=1}^N D_i\right)^{2k}\right] \lesssim  \left(N^{-1} \lor \E[D_1^2] \right)^k,
\end{equation}
where $\lesssim$ is hiding some constant depending only on $k$. To see this, consider the cases $\mathbb{E}[D_1^{2k}] \leq 1/N$ and $\mathbb{E}[D_1^{2k}] > 1/N$ separately and use that the second case implies 
$$
N \mathbb{E}[D_1^{2k}] < N^k \mathbb{E}[D_1^{2k}]^k \leq N^k \mathbb{E}[D_1^2]^k.
$$
Now fix any $0 \le u' < u \le 1 $ and $v \in [\eps,1-\eps]$ and write 
  $$
    D_i = \textbf{1}\left(\Phi_{v(1-v)}(B_i(v)) \leq u\right) - u - \textbf{1}\left(\Phi_{v(1-v)}(B_i(v)) \leq u'\right) + u'.
  $$
  In this case, $\mathbb{E}[D_i^2]$ is bounded by some fixed multiple of $|u-u'|$ for all $i = 1, \ldots, N$. By Eq.\@ \eqref{eq:2k_moments_sum_bound}, 
  \begin{equation}
      \label{eq:diffbound_1}
  \left\|\sqrt{N} \left\{\tilde{F}_{N,v}(u) - u - \tilde{F}_{N,v}(u') + u'\right\}\right\|_{L_{2k}}  = \mathbb{E}\left[\left(N^{-1/2} \sum_{i=1}^N D_i\right)^{2k}\right]^{1/(2k)} \lesssim \left(N^{-1} \lor |u-u'|\right)^{1/2}
  \end{equation}
  uniformly in $v \in (0,1)$. By setting $D_i = \textbf{1}\left(\Phi_{v(1-v)}(B_i(v)) \leq u\right) - u$ instead we also obtain
  \begin{equation}
    \label{eq:diffbound_2}
    \sup_{0 \leq v \leq 1} \left\|\sqrt{N} \left\{\tilde{F}_{N,v}(u) - u \right\}\right\|_{L_{2k}} \lesssim \left[N^{-1} \lor \min\{u, (1-u)\}\right]^{1/2} \lesssim \left[N^{-1} \lor \{u(1-u)\}\right]^{1/2}
  \end{equation}
for any $0 \leq u \leq 1$.

Now first assume that $u \in \{0,1\}, u' \neq u$. Then
\begin{align*}
&\left\|\{u(1-u)\}^{-\alpha} \sqrt{N}\left\{\tilde{F}_{N,v}(u) - u\right\} - \{u'(1-u')\}^{-\alpha} \sqrt{N}\left\{\tilde{F}_{N,v'}(u') - u'\right\}\right\|_{L_{2k}}
\\
& = 
\left\|\{u'(1-u')\}^{-\alpha} \sqrt{N}\left\{\tilde{F}_{N,v'}(u') - u'\right\}\right\|_{L_{2k}}
\end{align*}
and the claim follows from \eqref{eq:diffbound_2}. In what follows we thus assume $u,u' \in (0,1)$. Combining  Eqs.\@ \eqref{eq:diffbound_1} and \eqref{eq:diffbound_2} gives
 \begin{align}
    \begin{split}
      \label{eq:l2k_norm_v_gleich_1}
      & \left\|\{u(1-u)\}^{-\alpha} \sqrt{N}\left\{\tilde{F}_{N,v}(u) - u\right\} - \{u'(1-u')\}^{-\alpha} \sqrt{N}\left\{\tilde{F}_{N,v}(u') - u'\right\}\right\|_{L_{2k}} \\
      & \quad \leq \left|\{u(1-u)\}^{-\alpha} - \{u'(1-u')\}^{-\alpha}\right| \left\|\sqrt{N}\left\{\tilde{F}_{N,v}(u) - u\right\}\right\|_{L_{2k}}                         \\
      & \qquad + \{u'(1-u')\}^{-\alpha} \left\|\sqrt{N}\left\{\tilde{F}_{N,v}(u) - u\right\} - \sqrt{N}\left\{\tilde{F}_{N,v}(u') - u'\right\}\right\|_{L_{2k}}             \\
      & \quad \lesssim \left|\{u(1-u)\}^{-\alpha} - \{u'(1-u')\}^{-\alpha}\right|  \left[N^{-1} \lor \{u(1-u)\}\right]^{1/2} + \{u'(1-u')\}^{-\alpha} \left(N^{-1} \lor |u-u'|\right)^{1/2}.
    \end{split}
  \end{align}
  Let $f(u) = \{u(1-u)\}^{-\alpha}$, then $f'(u) = -\alpha (1-2u)\{u(1-u)\}^{-\alpha-1}$. By the mean value theorem, there exists some $\xi \in [u', u]$ such that $f(u) - f(u') = (u-u')f'(\xi)$. Because $|f'|$ is antitone on $(0,1/2]$ and isotone on $[1/2, 1)$, this means that
  \begin{align}
    \begin{split}
      \label{eq:fstrich_fallunterscheidung}
      \left|\{u(1-u)\}^{-\alpha} - \{u'(1-u')\}^{-\alpha}\right| & = |f'(\xi) (u - u')|                         \\
      & \leq \max\{|f'(u)|, |f'(u')|\} |u-u'|              \\
      & \leq \alpha \min\{u(1-u), u'(1-u')\}^{-\alpha-1} |u-u'|.
    \end{split}
  \end{align}
Thus, continuing the bound in~\eqref{eq:l2k_norm_v_gleich_1} we obtain for $\alpha \le 1$
\begin{multline}\label{eq:ES1}
\eqref{eq:l2k_norm_v_gleich_1} \leq  \alpha \left[N^{-1} \lor \{u(1-u)\}\right]^{1/2} \min\{u(1-u), u'(1-u')\}^{-\alpha-1} |u-u'|
\\
+ \{u'(1-u')\}^{-\alpha} \left(N^{-1} \lor |u-u'|\right)^{1/2}.
\end{multline}
The same bound holds with the roles of $u,u'$ switched.

Assume without loss of generality that $u(1-u) \le u'(1-u')$ and note that it suffices to consider two cases:
\begin{enumerate}
    \item\label{it:case1} $\max\{u(1-u),u'(1-u')\} \le 2|u-u'|$, and 
    \item\label{it:case2} $|u-u'| \le \min\{u(1-u),u'(1-u')\}$. 
\end{enumerate}
Indeed, if $|u-u'| \le \min\{u(1-u),u'(1-u')\}$ fails, we must have $u(1-u) \le |u-u'|$, but then by Lipschitz continuity of $u \mapsto u(1-u)$ on $[0,1]$ we have 
\[
u'(1-u') \le u(1-u) + |u-u'| \le 2|u-u'|,
\]
so that we are in case \ref{it:case1}. In the remainder of the proof, all constants hidden in $\lesssim, \gtrsim$ may depend on $\alpha$ but not on $N,v,u,u'$.

We start by considering the case $\max\{u(1-u),u'(1-u')\} \le 2|u-u'|$. Then by the triangle inequality
  \begin{align}
    \begin{split}
      \label{eq:l2k_norm_v_gleich_3_neu}
      & \left\|\{u(1-u)\}^{-\alpha} \sqrt{N}\left\{\tilde{F}_{N,v}(u) - u\right\} - \{u'(1-u')\}^{-\alpha} \sqrt{N}\left\{\tilde{F}_{N,v}(u') - u'\right\}\right\|_{L_{2k}} \\
      & \quad \leq \left\|\{u(1-u)\}^{-\alpha} \sqrt{N}\left\{\tilde{F}_{N,v}(u) - u\right\}\right\|_{L_{2k}} + \left\|\{u'(1-u')\}^{-\alpha} \sqrt{N}\left\{\tilde{F}_{N,v}(u') - u'\right\}\right\|_{L_{2k}}
      \end{split}
      \end{align}
      Assume first that $\{u(1-u)\} \geq N^{-1}$, then by Eq.\@ \eqref{eq:diffbound_2} 
      \begin{align}
      \begin{split}
          \label{eq:simple_1}
          \left\|\{u(1-u)\}^{-\alpha} \sqrt{N}\left\{\tilde{F}_{N,v}(u) - u\right\}\right\|_{L_{2k}} &\leq\{u(1-u)\}^{-\alpha}\left[N^{-1} \lor \{u(1-u)\}\right]^{1/2} \\
          &= \{u(1-u)\}^{1/2-\alpha} \\
          &\lesssim |u-u'|^{1/2 - \alpha},
      \end{split}
      \end{align}
      since $\max\{u(1-u),u'(1-u')\} \le 2|u-u'|$ by assumption. On the other hand, if $\{u(1-u)\} < N^{-1}$, then $Nu(1-u) < 1$, and in this case the Rosenthal inequality \eqref{holger1} implies 
      $$
\left\|\sqrt{N}\left\{\tilde{F}_{N,v}(u) - u\right\}\right\|_{L_{2k}}^{2k} \lesssim N^{-k} \max\left\{N u(1-u), \left(N u(1-u)\right)^k\right\} = N^{1-k} u(1-u).
$$
Therefore,
\begin{align}
\begin{split}
    \label{eq:simple_2}
\left\|\{u(1-u)\}^{-\alpha}\sqrt{N}\left\{\tilde{F}_{N,v}(u) - u\right\}\right\|_{L_{2k}} &\lesssim \{u(1-u)\}^{1/(2k)-\alpha} N^{(1-k)/(2k)} \\
&\leq (N^{-1})^{1/(2k) - \alpha + k/(2k) - 1/(2k)} = N^{-1/2 + \alpha}.
\end{split}
\end{align}
if $\alpha < 1/(2k)$. Eqs.\@ \eqref{eq:simple_1} and \eqref{eq:simple_2} mean that we always have the bound
$$
\left\|\{u(1-u)\}^{-\alpha} \sqrt{N}\left\{\tilde{F}_{N,v}(u) - u\right\}\right\|_{L_{2k}} \lesssim \left[ N^{-1} \lor |u-u'| \right]^{1/2-\alpha},
$$
and the same argument can be made for $u'$, resulting in the same bound. Eq.\@ \eqref{eq:l2k_norm_v_gleich_3_neu} therefore reduces to
\begin{align*}
    &\left\|\{u(1-u)\}^{-\alpha} \sqrt{N}\left\{\tilde{F}_{N,v}(u) - u\right\} - \{u'(1-u')\}^{-\alpha} \sqrt{N}\left\{\tilde{F}_{N,v}(u') - u'\right\}\right\|_{L_{2k}} \\
    &\lesssim  \left[ N^{-1} \lor |u-u'| \right]^{1/2-\alpha},
\end{align*}
which concludes the discussion for case \ref{it:case1}.

Next, consider case \ref{it:case2}, i.e.\@ $|u-u'| \le \min\{u(1-u),u'(1-u')\}$. If also $u(1-u) \ge 1/N,  u'(1-u')\ge 1/N$ then by Eq. \eqref{eq:ES1}
  \begin{align*}
    \begin{split}
      & \left\|\{u(1-u)\}^{-\alpha} \sqrt{N}\left\{\tilde{F}_{N,v}(u) - u\right\} - \{u'(1-u')\}^{-\alpha} \sqrt{N}\left\{\tilde{F}_{N,v}(u') - u'\right\}\right\|_{L_{2k}} \\
      \lesssim~&   \{u(1-u)\}^{1/2}\{u(1-u)\}^{-\alpha-1} |u-u'|^{\alpha+1/2}|u-u'|^{1/2-\alpha} + \{u'(1-u')\}^{-\alpha} \big[ N^{-1} \lor |u-u'| \big]^{1/2}
      \\
      \lesssim~&  \big[ N^{-1} \lor |u-u'| \big]^{1/2-\alpha}.
    \end{split}
  \end{align*}
If $u(1-u) \le 1/N$, then $u \in [0,2/N]\cup[1-2/N,1]$. Then $|u-u'| \le u(1-u)$ implies $u' \in [0,3/N]\cup[1-3/N,1]$ and so $u'(1-u') \lesssim 1/N$. Then by the same argument as in Eq.\@ \eqref{eq:simple_2}, we get
\[
\left\|\{u(1-u)\}^{-\alpha} \sqrt{N}\left\{\tilde{F}_{N,v}(u) - u\right\} - \{u'(1-u')\}^{-\alpha} \sqrt{N}\left\{\tilde{F}_{N,v}(u') - u'\right\}\right\|_{L_{2k}} \lesssim N^{-1/2+\alpha}.
\]
The case $u(1-u) > 1/N, u'(1-u') \le 1/N$ is ruled out since $u(1-u) \le u'(1-u')$ by assumption, and this completes our discussion of case \ref{it:case2}. 

It therefore always holds that
\begin{align}
    \begin{split}
      \label{eq:l2k_norm_v_gleich_4}
      &\sup_{0 \leq v \leq 1} \left\|\{u(1-u)\}^{-\alpha} \sqrt{N}\left\{\tilde{F}_{N,v}(u) - u\right\} - \{u'(1-u')\}^{-\alpha} \sqrt{N}\left\{\tilde{F}_{N,v}(u') - u'\right\}\right\|_{L_{2k}} \\
      &\quad \lesssim \max\left\{N^{-1/2 + \alpha}, |u-u'|^{1/2 - \alpha}\right\}.
    \end{split}
  \end{align}
provided that  $\alpha < 1/(2k)$, where the constant in $\lesssim$ is uniform in $u,u' \in [0,1]$.

  Writing $\Delta_{v,v'} = B_1(v')/\sqrt{v'(1-v')} - B_1(v)/\sqrt{v(1-v)}$, we have for any $\varepsilon \leq v, v' \leq 1 - \varepsilon$ and any $0 \leq u \leq 1$
  \begin{align}
\nonumber
      & \mathbb{E}\left[\left\{\textbf{1}\left(\Phi_{v(1-v)}(B_1(v)) \leq u\right) - \textbf{1}\left(\Phi_{v'(1-v')}(B_1(v)) \leq u\right)\right\}^2\right]                                                           \\
      & 
      \nonumber 
      \quad = \mathbb{E}\left[\left\{\textbf{1}\left(\frac{B_1(v)}{\sqrt{v(1-v)}} \leq \Phi^{-1}(u)\right) - \textbf{1}\left(\frac{B_1(v')}{\sqrt{v'(1-v')}} \leq \Phi^{-1}(u)\right)\right\}^2\right]              \\
      & 
      \nonumber
      \quad = \mathbb{E}\left[\left\{\textbf{1}\left(\frac{B_1(v)}{\sqrt{v(1-v)}} \leq \Phi^{-1}(u)\right) - \textbf{1}\left(\frac{B_1(v)}{\sqrt{v(1-v)}} + \Delta_{v,v'} \leq \Phi^{-1}(u)\right)\right\}^2\right] \\
      & \nonumber
      \quad  \leq \mathbb{E} \left|\textbf{1}\left(\frac{B_1(v)}{\sqrt{v(1-v)}} \leq \Phi^{-1}(u)\right) - \textbf{1}\left(\frac{B_1(v)}{\sqrt{v(1-v)}} + \Delta_{v,v'} \leq \Phi^{-1}(u)\right)\right|             \\
      & \quad \leq \mathbb{P}\left(\frac{B_1(v)}{\sqrt{v(1-v)}} \in \left[\Phi^{-1}(u) - |\Delta_{v,v'}|, \Phi^{-1}(u) + |\Delta_{v,v'}|\right]\right).
        \label{eq:indikator_differenz}
  \end{align}
Assume without loss of generality that $v \leq v'$. By elementary calculations, we see that
  \begin{align}
    \begin{split}
\label{eq:identities_variance_covariance_correlation}
      & \sigma^2_{v,v'} = \mathrm{Var}(\Delta_{v,v'}) = 2\left(1 - \sqrt{\frac{v(1-v')}{(1-v)v'}}\right),  \\
      & \mathrm{Cov}\left(\frac{B_1(v)}{\sqrt{v(1-v)}}, \Delta_{v,v'}\right) = -\frac{\sigma^2_{v,v'}}{2}, \\
      & \mathrm{Cor}\left(\frac{B_1(v)}{\sqrt{v(1-v)}}, \Delta_{v,v'}\right) = -\frac{\sigma_{v,v'}}{2}.   \\
    \end{split}
  \end{align}
  Now consider the following general observation: For any $x \geq 1$, Mills Ratio implies 
  $$
  \varphi(x) \leq 2 x (1 - \Phi(x)).
  $$
  Let $\kappa > 0$ be the smallest number such that $|\Phi^{-1}(u)| \geq 1$ for all $u \notin (1/2 - \kappa, 1/2 + \kappa)$. If $u \leq 1/2 - \kappa$, then for any $r < 1$
  $$
    \varphi\left(\Phi^{-1}(u)\right) = \varphi\left(\Phi^{-1}(1-u)\right) \leq 2 \Phi^{-1}(1-u) u \leq c u^r \leq 2c \{u(1-u)\}^r
  $$
  for a constant $c = c(r)$ depending on $r$. Similarly, if $u \geq 1/2 + \kappa$,
  $$
    \varphi\left(\Phi^{-1}(u)\right) \leq 2 \Phi^{-1}(u) (1-u) \leq c (1-u)^r \leq 2c \{u(1-u)\}^r.
  $$
  Now if $u \in (1/2 - \kappa, 1/2 + \kappa)$, then trivially
  $$
    \varphi\left(\Phi^{-1}(u)\right) \leq 1 = \frac{\{u(1-u)\}^r}{\{u(1-u)\}^r} \leq \left(\frac{1}{2} - \kappa\right)^{-2r} \{u(1-u)\}^r.
  $$
  Thus, we can find some constant $C = C(r)$ depending only on $r$ such that $\varphi\left(\Phi^{-1}(u)\right) \leq C \{u(1-u)\}^r$ for any $0 \leq u \leq 1$. We can now use Lemma \ref{lem:gaussian_lemma} to bound the final term in Eq.\@ \eqref{eq:indikator_differenz} by
  \begin{equation}
  \label{eq:delta_v_unterschiedlich}
    \frac{\sigma_{v,v'} \varphi\left(\Phi^{-1}(u)/5\right)}{\sqrt{1 - \sigma_{v,v'}^2/4}}  \lesssim \frac{\sigma_{v,v'} \varphi\left(\Phi^{-1}(u)\right)^{1/25}}{\sqrt{1 - \sigma_{v,v'}^2/4}} \lesssim \frac{\sigma_{v,v'} \left\{u (1-u)\right\}^{1/26}}{\sqrt{1 - \sigma_{v,v'}^2/4}},
  \end{equation}
  modulo some universal multiplicative constants where $\lesssim$ is hiding universal constants. Let us now distinguish three cases:
  \begin{enumerate}
      \item \label{it:v_bound_large} The bound in Eq.\@ \eqref{eq:delta_v_unterschiedlich} is $\geq N^{-1}$,
      \item \label{it:v_bound_small_u_middle} The bound in Eq.\@ \eqref{eq:delta_v_unterschiedlich} is $< N^{-1}$, and $u \in \left[N^{-1}, 1 - N^{-1}\right]$,
       \item \label{it:v_bound_small_u_extreme} The bound in Eq.\@ \eqref{eq:delta_v_unterschiedlich} is $< N^{-1}$, and $u \notin \left[N^{-1}, 1 - N^{-1}\right]$.
  \end{enumerate}

  \paragraph{Case \ref{it:v_bound_large}.}
  By Eq.\@ \eqref{eq:2k_moments_sum_bound},
  \begin{equation}
    \label{eq:diffbound_v}
    \sup_{u \in [0,1]}\left\|\sqrt{N} \left\{\tilde{F}_{N,v}(u) - u - \tilde{F}_{N,v'}(u) + u\right\}\right\|_{L_{2k}}  \lesssim \sqrt{\sigma_{v,v'}} \{u(1-u)\}^{1/52} \left(1 - \sigma_{v,v'}^2/4\right)^{-1/4}.
  \end{equation}
Thus,
  \begin{align*}
      &\sup_{0 \leq u \leq 1} \left\|\{u(1-u)\}^{-\alpha} \sqrt{N}\left\{\tilde{F}_{N,v}(u) - u\right\} - \{u(1-u)\}^{-\alpha} \sqrt{N}\left\{\tilde{F}_{N,v'}(u) - u\right\}\right\|_{L_{2k}} \\
      &\quad \lesssim  \sqrt{\sigma_{v,v'}} \left(1 - \sigma_{v,v'}^2/4\right)^{-1/4}
  \end{align*}
  for all $\varepsilon \leq v,v' \leq 1-\varepsilon$, provided $\alpha < 1/52$.

  \paragraph{Case \ref{it:v_bound_small_u_middle}.}
  We have 
  $$
  \sup_{u \in [0,1]}\left\|\sqrt{N} \left\{\tilde{F}_{N,v}(u) - u - \tilde{F}_{N,v'}(u) + u\right\}\right\|_{L_{2k}}  \lesssim  N^{-1/2}
  $$
  and $\{u(1-u)\}^{-\alpha} \lesssim N^\alpha$, whence 
  $$
  \sup_{0 \leq u \leq 1} \left\|\{u(1-u)\}^{-\alpha} \sqrt{N}\left\{\tilde{F}_{N,v}(u) - u\right\} - \{u(1-u)\}^{-\alpha} \sqrt{N}\left\{\tilde{F}_{N,v'}(u) - u\right\}\right\|_{L_{2k}} \lesssim N^{-1/2+\alpha}.
  $$

  \paragraph{Case \ref{it:v_bound_small_u_extreme}.}
Observing  Rosenthal's inequality \eqref{holger1}
and  $
Nu(1-u) \leq 1$, we have 
$$
\left\|\sqrt{N}\left\{\tilde{F}_{N,v}(u) - u\right\}\right\|_{L_{2k}}^{2k} \lesssim N^{-k} \max\left\{N u(1-u), \left(N u(1-u)\right)^k\right\} = N^{1-k} u(1-u),
$$
and so
\begin{align*}
\left\|\{u(1-u)\}^{-\alpha}\sqrt{N}\left\{\tilde{F}_{N,v}(u) - u\right\}\right\|_{L_{2k}} &\lesssim \{u(1-u)\}^{1/(2k)-\alpha} N^{(1-k)/(2k)} \\
&\leq (N^{-1})^{1/(2k) - \alpha + k/(2k) - 1/(2k)} = N^{-1/2 + \alpha}.
\end{align*}
if $\alpha < 1/(2k)$. By the triangle inequality, this gives 
$$
\sup_{0 \leq v \leq 1} \left\|\{u(1-u)\}^{-\alpha} \sqrt{N}\left\{\tilde{F}_{N,v}(u) - u\right\} - \{u'(1-u')\}^{-\alpha} \sqrt{N}\left\{\tilde{F}_{N,v}(u') - u'\right\}\right\|_{L_{2k}} \lesssim N^{-1/2 + \alpha}.
$$

  These three cases combine to give us 
    \begin{align}
    \begin{split}
      \label{eq:l2k_norm_u_gleich}
      &\sup_{0 \leq u \leq 1} \left\|\{u(1-u)\}^{-\alpha} \sqrt{N}\left\{\tilde{F}_{N,v}(u) - u\right\} - \{u(1-u)\}^{-\alpha} \sqrt{N}\left\{\tilde{F}_{N,v'}(u) - u\right\}\right\|_{L_{2k}} \\
      &\quad \lesssim  \max\left\{N^{-1/2 + \alpha}, \sqrt{\sigma_{v,v'}} \left(1 - \sigma_{v,v'}^2/4\right)^{-1/4}\right\}
    \end{split}
  \end{align}
  if $\alpha < \min\{1/(2k), 1/52\}$.
  
  Eqs.\@ \eqref{eq:l2k_norm_v_gleich_4} and \eqref{eq:l2k_norm_u_gleich} together imply
  \begin{align}
    \begin{split}
      \label{eq:l2k_norm_nichts_gleich}
      &\left\|\{u(1-u)\}^{-\alpha} \sqrt{N}\left\{\tilde{F}_{N,v}(u) - u\right\} - \{u'(1-u')\}^{-\alpha} \sqrt{N}\left\{\tilde{F}_{N,v'}(u') - u'\right\}\right\|_{L_{2k}} \\
      &\quad \lesssim \max\left\{N^{-1/2 + \alpha}, |u-u'|^{1/2 - \alpha} + \sqrt{\sigma_{v,v'}} \left(1 - \sigma_{v,v'}^2/4\right)^{-1/4}\right\}
    \end{split}
  \end{align}
  if $\alpha < \min\{1/(2k), 1/52\}$, where the constant in $\lesssim$ is uniform in $u,u' \in [0,1], v, v' \in [\eps,1-\eps]$.

  By the mean value theorem we have $2|1 - \sqrt{1 + x}| \leq |x|(1-|x|)^{-1/2}$ on $(-1,1)$. Since
  $$
    1 + \frac{v - v'}{v'(1-v)} = \frac{v(1-v')}{v'(1-v)}
  $$
we can use the first identity in Eq.\@ \eqref{eq:identities_variance_covariance_correlation} to see that
$$
\sigma_{v,v'}^2 \leq |v-v'| \frac{1}{v'(1-v)} \left(1 - \frac{|v - v'|}{v'(1-v)}\right)^{-1/2}.
$$
If $|v-v'| \leq v'(1-v)/2$, then this implies 
\begin{equation}
      \label{eq:sigma2_formula_1}
\sigma_{v,v'}^2 \leq |v-v'| \frac{\sqrt{2}}{v'(1-v)}.
\end{equation}
On the other hand, if $|v-v'| > v'(1-v)/2$, we can use the trivial bound $\sigma_{v,v'}^2 \leq 2$ to obtain 
\begin{equation}
      \label{eq:sigma2_formula_2}
\sigma_{v,v'}^2 \leq |v-v'| \frac{2}{|v-v'|} < |v-v'| \frac{4}{v'(1-v)},
\end{equation}
and in fact this inequality holds for all $v,v' \in (0,1)$ since it is less sharp than \eqref{eq:sigma2_formula_1}. Here, we are considering $\varepsilon \leq v,v' \leq 1-\varepsilon$, and so $v'(1-v) \geq \varepsilon^2$, which means that Eq.\@ \eqref{eq:sigma2_formula_2} reduces to $\sigma_{v,v'}^2 \leq c(\varepsilon) |v-v'|$ with $c(\varepsilon) = 4\varepsilon^{-2}$. Using this bound in Eq.\@ \eqref{eq:l2k_norm_nichts_gleich} yields
  \begin{align*}
     & \left\|\{u(1-u)\}^{-\alpha} \sqrt{N}\left\{\tilde{F}_{N,v}(u) - u\right\} - \{u'(1-u')\}^{-\alpha} \sqrt{N}\left\{\tilde{F}_{N,v'}(u') - u'\right\}\right\|_{L_{2k}} \\
     & \quad \lesssim \max\left\{N^{-1/2 + \alpha}, |u-u'|^{1/2 - \alpha} + c(\varepsilon) |v-v'|^{1/4}\right\},
  \end{align*}
  if $\alpha < \min\{1/(2k), 1/52\}$, which is what we wanted to prove.
\end{proof}

\begin{lemma}\label{lem:approxQNQn'}
  If $a_N \to 0$ and $N a_N \to \infty$ for $N \to \infty$, then it follows for the processes $Q_N$ and $Q_N'$ in \eqref{holger2} that
  \begin{equation}   \label{eq:QN_QNstrich_abstand}
    \sup_{v \in [0,1]} \sup_{a_N \leq u \leq 1-a_N} \{u(1-u)\}^{-\alpha} \left| Q_N(u,g_N(v)) + Q_N'(u,g_N(v))\right| = o_\mathbb{P}(1)
  \end{equation}
for all $0 < \alpha < 1/4$.    
\end{lemma}

\begin{proof}[Proof of Lemma~\ref{lem:approxQNQn'}]
We begin with some preliminary observations. If $v(1-v) = 0$, then $Q_N(u,g_N(v)) = Q_N'(u,g_N(v)) = 0$ for all $u$, so in what follows it suffices to consider the case $v \in (0,1)$. Then also $g_N(v) \in (0,1)$. Let $U_1, \ldots, U_N$ be i.i.d. random variables with uniform distribution on the unit interval and denote by $J_N$ their empirical distribution function. Then, for any fixed $v \in (0,1)$, $(Q_N(u, v))_{0 \leq u \leq 1}$ is equal in distribution to
  $$
    \left\{\sqrt{N}\sqrt{v(1-v)} \left(J_N^{-1}(u) - u\right)\right\}_{0 \leq u \leq 1}.
  $$
  By Eq.\@ (3.1.11) in \cite{shorack_wellner:1986} we have
  $$
    \sup_{0 \leq u \leq 1} |J_N^{-1}(u) - u| = \sup_{0 \leq u \leq 1} |J_N(u) - u|,
  $$
  and the Dvoretzky-Kiefer-Wolfowitz (DKW) inequality \citep[e.g.\@ Inequality 1 in Chapter 9, Section 2, in][]{shorack_wellner:1986} implies
  $$
    \mathbb{P}\left(\sup_{0 \leq u \leq 1} \left|J_N(u) - u\right| \geq N^{-1/2}\, t \right) \leq 58 \, \exp\left(-2 t^2\right)
  $$
  for all $t \geq 0$. Thus,
  \begin{align}
    \begin{split}
      \label{eq:3logN_bound}
      & \mathbb{P}\left(\max_{j = 1, \ldots, N^2-1} \sup_{0 \leq u \leq 1} \left|\tilde{F}_{N,v_j}^{-1}(u) - u\right| > N^{-1/2} \sqrt{3 \log N}\right) \\
      & \quad \le \sum_{j=1}^{N^2-1} \mathbb{P}\left(\sup_{0 \leq u \leq 1} \left|\tilde{F}_{N,v_j}^{-1}(u) - u\right| > N^{-1/2} \sqrt{3 \log N}\right)  \\
      & \quad = (N^2-1) \, \mathbb{P}\left(\sup_{0 \leq u \leq 1} \left|J_{N}(u) - u\right| > N^{-1/2} \sqrt{3 \log N}\right)                             \\
      & \quad \leq 58 \, N^{-4} = o(1).
    \end{split}
  \end{align}
  Finally, by Lemma 21.1 in \cite{vandervaart:asymptotic_statistics}, $J_N \circ J_N^{-1} (u) \geq u$ for all $0 < u < 1$ with equality if and only if $u \in J_N(\mathbb{R}) = \{1/N, \ldots, N/N\}$. This implies that $|J_N \circ J_N^{-1}(u) - u| \leq 1/N$ for all $u$. Hence, on the event $\sup_{0 \leq u \leq 1} \left|J_N^{-1}(u) - u\right| \le N^{-1/2} \sqrt{3 \log N}$,
  \begin{align*}
    \left|J_N(u) + J_N^{-1}(u) - 2u\right| & = \left|\left\{J_N(u) - u\right\} - \left\{J_N\left(J_N^{-1}(u)\right) - J_N^{-1}(u)\right\} + \left\{J_N\left(J_N^{-1}(u)\right) - u\right\} \right| \\
                                      & \leq \frac{1}{N} + \left|\left\{J_N(u) - u\right\} - \left\{J_N\left(J_N^{-1}(u)\right) - J_N^{-1}(u)\right\}\right|                                  \\
    & \leq \frac{1}{N} + \sup_{|u-u'|\leq N^{-1/2} \sqrt{3 \log N}} \left|\left\{J_N(u) - u\right\} - \left\{J_N\left(u'\right) - u'\right\}\right|.
  \end{align*}

Applying this to $Q_N$ and $Q_N'$, we see that Eq.\@ \eqref{eq:QN_QNstrich_abstand} is bounded by
  \begin{equation}
    \label{eq:QN_QNstrich_abstand_2}
    (a_N/2)^{-\alpha} \left[ \frac{1}{\sqrt{N}} + \max_{j = 1, \ldots, N^2-1} \sup_{|u-u'|\leq N^{-1/2} \sqrt{3 \log N}} \sqrt{N} \left| \left\{\tilde{F}_{N,v_j}(u) - u\right\} - \left\{\tilde{F}_{N,v_j}(u') - u'\right\}\right|\right] + o_\mathbb{P}(1)
  \end{equation}
  for $N$ large enough such that $a_N \leq 1/2$. 
  
  Furthermore, by Theorem 2 in Chapter 12, Section 1, in \cite{shorack_wellner:1986}, there exist constants $c_1, c_2, c_3 > 0$ and a sequence $(W_N)_{N \in \mathbb{N}}$ of standard Brownian bridges $W_N$ such that
  $$
    \mathbb{P}\left(\sup_{0 \leq u \leq 1}\left|\sqrt{N}\left\{J_N(u) - u\right\} - W_N(u)\right| > \frac{(c_1 \log N + x){\log N}}{\sqrt{N}}\right) \leq c_2 \exp(-c_3 x)
  $$
  for all $x > 0$. This is the Hungarian embedding of the uniform empirical process. Defining 
\begin{align}
    \label{hd10}
\Delta_N = \{\eps a_N^\alpha-N^{-1/2}c_2(\log N)^2\}/\log N 
\end{align}
we have
  \begin{align}
  \begin{split}
      \label{eq:kmt_approximation}
     & \mathbb{P}\left(\sup_{0 \leq u \leq 1} a_N^{-\alpha} \left|\sqrt{N}\left\{J_N(u) - u\right\} - W_N(u)\right| > \varepsilon\right)                               \\
     & \quad = \mathbb{P}\left(\sup_{0 \leq u \leq 1}\left|\sqrt{N}\left\{J_N(u) - u\right\} - W_N(u)\right| > \frac{c_1 \log N + \sqrt{N}\Delta_N}{\sqrt{N}}\right) \\
     & \quad \leq c_2 \exp\left(-c_3 \sqrt{N}\Delta_N\right). 
     \end{split}
  \end{align}
  Now let $W_N^{(j)}$, $j = 1, \ldots, N^2-1$, $N \in \mathbb{N}$ denote the Brownian bridges approximating $\sqrt{N}\left\{\tilde{F}_{N,v_j}(u) - u\right\}$ in this sense. By a union bound argument, we have
  \begin{align}
  \begin{split}
\label{eq:op_union}
  &\mathbb{P}\left(\max_{j = 1, \ldots, N^2-1} \sup_{|u-u'| \leq \sqrt{3 N^{-1} \log N}} a_N^{-\alpha}\sqrt{N} \left| \left\{\tilde{F}_{N,v_j}(u) - u\right\} - \left\{\tilde{F}_{N,v_j}(u') - u'\right\}\right| > 3\varepsilon\right) \\
  &\quad\leq 2 N^2 \max_{j = 1, \ldots, N^2-1}\mathbb{P}\left(\sup_{0 \leq u \leq 1}a_N^{-\alpha}\left|\sqrt{N}\left\{\tilde{F}_{N,v_j}(u) - u\right\} - W_N^{(j)}(u)\right| > \varepsilon\right) \\
  &\qquad+ N^2\max_{j = 1, \ldots, N^2-1}\mathbb{P}\left(\sup_{|u-u'| \leq \sqrt{3 N^{-1} \log N}} a_N^{-\alpha}\left|W_N^{(j)}(u) - W_N^{(j)}(u')\right| > \varepsilon\right).
  \end{split}
  \end{align}
  The first summand is bounded by $2c_2 N^2 \exp(-c_3 \sqrt{N} \Delta_N)$ by Eq.\@ \eqref{eq:kmt_approximation}, and this bound converges to $0$ for $N \to \infty$ since $\sqrt{N}\Delta_N \gg \log N$ by the fact that $\alpha < 1/4$ and $a_N \gg N^{-1}$. On the other hand, we know from Lemma \ref{lem:brownian_increments_orlicz} that
  $$
  \left\|a_N^{-\alpha} \sup_{|u-u'| \leq \sqrt{3 N^{-1} \log N}} \left|W_N^{(j)}(u) - W_N^{(j)}(u')\right|\right\|_{\psi_2} \lesssim  a_N^{-\alpha} (N^{-1} \log N)^{r/4}
  $$
  for all $r < 1$, where $\lesssim$ is hiding a constant depending only on $r$. Using the tail inequality
  $$
  \mathbb{P}(|X| > x) \leq \frac{1}{\psi_2(x/\|X\|_{\psi_2})} 
  $$
  for any random variable $X$, we obtain for sufficiently large $N$
  \begin{align}
  \begin{split}
      \label{eq:brownian_bridge_log_increment}
      N^2\mathbb{P}\left(\sup_{|u-u'| \leq \sqrt{3 N^{-1} \log N}} a_N^{-\alpha}\left|W_N^{(j)}(u) - W_N^{(j)}(u')\right| > \varepsilon\right) \lesssim N^2\exp\left[-\varepsilon^2 a_N^{2\alpha} \left(\frac{N}{\log N}\right)^{r/2}\right],
      \end{split}
  \end{align}
  and the bound on the right-hand side converges to $0$ if $\alpha < r/4$ since then $a_N^{2\alpha} \big (\frac{N}{\log N}\big )^{r/2} =\big (a_N^{4\alpha/r}\frac{N}{\log N}\big )^{r/2} \gg \big (N^{-4\alpha/r}\frac{N}{\log N}\big )^{r/2}$. In light of Eq.\@ \eqref{eq:op_union}, we see that Eq.\@ \eqref{eq:QN_QNstrich_abstand_2} is $o_\mathbb{P}(1)$ for all $\alpha < 1/4$, and this also proves Eq.\@ \eqref{eq:QN_QNstrich_abstand}.   
\end{proof}

\begin{lemma}\label{lem:discretQN'inu}
Assume $N a_N \to \infty$ for $N \to \infty$. Additionally assume that $a_N = O(N^{-\eta})$ for some $\eta >0$. Then for any fixed $\alpha \in (0,1/32]$ 
  \begin{equation}
    \label{eq:QN_prime_u_discrete}
    \sup_{0 \leq u, v \leq 1}  \left|\{u(1-u)\}^{-\alpha} Q_N'(u,g_N(v)) - \{g_N(u)(1-g_N(u))\}^{-\alpha} Q_N'(g_N(u),g_N(v))\right| = o_\mathbb{P}(1).
  \end{equation}
\end{lemma}

\begin{proof}[Proof of Lemma~\ref{lem:discretQN'inu}]
It suffices to prove the following claims
  \begin{equation}
    \label{eq:QN_prime_u_discrete-help}
    \sup_{0 \leq v \leq 1} \sup_{u \in [a_N,1-a_N]}  \left|\{u(1-u)\}^{-\alpha} Q_N'(u,g_N(v)) - \{g_N(u)(1-g_N(u))\}^{-\alpha} Q_N'(g_N(u),g_N(v))\right| = o_\mathbb{P}(1).
  \end{equation}
and
  \begin{align}
      \max_{j=1,\dots,N^2-1} \sup_{0 \le u \le 2a_N} |u^{-\alpha} Q_N'(u,v_j)| &= o_\P(1), \label{eq:truncneg1}
      \\
            \max_{j=1,\dots,N^2-1} \sup_{1-2a_N \le u \le 1} |u^{-\alpha} Q_N'(u,v_j)| &= o_\P(1). \label{eq:truncneg2}
  \end{align}
This is true since $u \in [0,a_N]$ implies $0 \le g_N(u) \le a_N + N^{-2} \le 2a_N$ and similarly $u \in [1-a_N,1]$ implies $g_N(u) \in [1-2a_N,1]$.

\bigskip 
 
\emph{Proof of~\eqref{eq:QN_prime_u_discrete-help}} For $j = 1, \ldots, N^2-1$, let $W_j$ denote a standard Brownian bridge such that
  \begin{equation}
    \label{eq:hungarian_embedding_tail_bound}
    \mathbb{P}\left(\sup_{0 \leq u \leq 1} \left| \sqrt{N}\left\{\tilde{F}_{N,v_j}(u) - u\right\} - W_j(u)\right| > \frac{{(c_1 \log N + x)\log N}}{\sqrt{N}}\right) \leq c_2 \exp(-c_3 x)
  \end{equation}
  for all $x > 0$. The existence of these Brownian bridges is guaranteed by {Theorem 2 in Chapter 12, Section 1,} in \cite{shorack_wellner:1986}. 
  Then
  \begin{align}
    \begin{split}
      \label{eq:diskretisierung_punktweise}
      & \mathbb{P}\left(\max_{j = 1, \ldots, N^2-1}\sup_{a_N \leq u \leq 1-a_N} \left|\{u(1-u)\}^{-\alpha} Q_N'(u,v_j) - \{g_N(u)(1-g_N(u))\}^{-\alpha} Q_N'(g_N(u),v_j)\right| > 4\varepsilon\right)                                                 \\
      & \quad \leq N^2 \max_{j = 1, \ldots, N^2-1} \mathbb{P}\left(\sup_{a_N \leq u \leq 1-a_N} \sqrt{v_j(1-v_j)}\left|\frac{W_j(u)}{\{u(1-u)\}^{\alpha}}  - \frac{W_j(g_N(u))}{\{g_N(u)(1-g_N(u))\}^{\alpha}} \right| > 2\varepsilon\right) 
      \\
      & \qquad + 2 \max_{j = 1, \ldots, N^2-1} N^2 \mathbb{P}\left(\sup_{a_N \leq u \leq 1-a_N}  {\{u(1-u)\}^{-\alpha}} \left|Q_N'(u,v_j) - \sqrt{v_j(1-v_j)}W_j(u)\right| > \varepsilon\right).
    \end{split}
  \end{align}
  Let us consider the first probability on the right-hand side. Use the triangle inequality to see that
\begin{align}
\begin{split}
    \label{eq:probs_eps_1}
    &N^2 \mathbb{P}\left(\sup_{a_N \leq u \leq 1-a_N} \sqrt{v_j(1-v_j)}\left|\{u(1-u)\}^{-\alpha} W_j(u) - \{g_N(u)(1-g_N(u))\}^{-\alpha} W_j(g_N(u))\right| > 2\varepsilon\right) \\
    &\leq N^2 \mathbb{P}\left(\sup_{a_N \leq u \leq 1-a_N} \sqrt{v_j(1-v_j)} \{u(1-u)\}^{-\alpha}\left|W_j(u) -  W_j(g_N(u))\right| > \varepsilon\right) \\
    &\quad + N^2\mathbb{P}\left(\sup_{a_N \leq u \leq 1-a_N} \sqrt{v_j(1-v_j)}\left|W_j(g_N(u))\left[\{u(1-u)\}^{-\alpha} - \{g_N(u)(1-g_N(u))\}^{-\alpha}\right]\right| > \varepsilon\right).
\end{split}
\end{align}
Since $|u - g_N(u)|\leq N^{-2} < \sqrt{3N^{-1}\log N}$ and $(1-a_N)^{-\alpha}\sqrt{v_j ( 1-v_j)} \leq 2^\alpha/2 = 2^{\alpha-1}$ for large $N$, it follows that
\begin{align}
\begin{split}
    \label{eq:probs_eps_2}
    &N^2 \mathbb{P}\left(\sup_{a_N \leq u \leq 1-a_N} \sqrt{v_j(1-v_j)} \{u(1-u)\}^{-\alpha}\left|W_j(u) -  W_j(g_N(u))\right| > \varepsilon\right) \\
    &\leq N^2 \mathbb{P}\left(\sup_{|u-u'| \leq \sqrt{3N^{-1}\log N}} a_N^{-\alpha}\left|W_j(u) -  W_j(g_N(u))\right| > \varepsilon 2^{1 - \alpha}\right)
\end{split}
\end{align}
for sufficiently large $N$, and the right-hand side converges to $0$ by Eq.\@ \eqref{eq:brownian_bridge_log_increment}. On the other hand, Eq.\@ \eqref{eq:fstrich_fallunterscheidung} gives us for $N$ sufficiently large so that $N^{-2} \ge a_N/2$ and $a_N \le 1/2$ which implies $g_N(u)(1-g_N(u)) \ge a_N/4$ for all $u \in [a_N,1-a_N]$,
$$
\left|\{u(1-u)\}^{-\alpha} - \{g_N(u)(1-g_N(u))\}^{-\alpha}\right| \leq \alpha 4^{1 + \alpha} a_N^{-\alpha - 1} N^{-2}
$$
for large $N$, and therefore (also for large $N$), noting that $v_j(1-v_j) \le 1$
\begin{align}
\begin{split}
    \label{eq:probs_eps_3}
& N^2\mathbb{P}\left(\sup_{a_N \leq u \leq 1-a_N} \sqrt{v_j(1-v_j)}\left|W_j(g_N(u))\left[\{u(1-u)\}^{-\alpha} - \{g_N(u)(1-g_N(u))\}^{-\alpha}\right]\right| > \varepsilon\right) \\
& \leq N^2\mathbb{P}\left(\sup_{a_N \leq u \leq 1-a_N} \left|W_j(g_N(u))\right| > \frac{\varepsilon}{\alpha 4^{1+\alpha}} N^2 a_N^{1+\alpha}\right) \\
& \leq N^2\mathbb{P}\left(\sup_{0 \leq u \leq 1} \left|W_j(u)\right| > \frac{\varepsilon}{\alpha 4^{1+\alpha}} N^2 a_N^{1+\alpha}\right) \\
& \leq 2 N^2 \exp\left\{-2\varepsilon/(\alpha 4^{1+\alpha}) N^2 a_N^{1+\alpha}\right\}
\end{split}
\end{align}
by Eq. (11) in Chapter 2, Section 2, of \cite{shorack_wellner:1986}. The expression in the last line converges to $0$ as $N \to \infty$, because $N^2 a_N^{1+\alpha} = N^{1-\alpha} (N a_N)^{1+\alpha}$, and $Na_N \to \infty$ by assumption. Eqs.\@ \eqref{eq:probs_eps_1}, \eqref{eq:probs_eps_2} and \eqref{eq:probs_eps_3} imply that
\begin{equation}
\label{eq:probs_convergence_1}
N^2 \max_{j = 1, \ldots, N^2-1}\mathbb{P}\left(\sup_{a_N \leq u \leq 1-a_N} \sqrt{v_j(1-v_j)}\left|\frac{W_j(u)}{\{u(1-u)\}^{\alpha}}  - \frac{W_j(g_N(u))}{\{g_N(u)(1-g_N(u))\}^{\alpha}} \right| > 2\varepsilon\right) \to 0
\end{equation}
as $N \to \infty$. For the second probability on the right-hand side of Eq.\@ \eqref{eq:diskretisierung_punktweise} we have for any $\varepsilon > 0$, since $N^{-1/2} \log N = o\left(a_N^\alpha\right)$ for $\alpha < 1/2$,
  \begin{align*}
     & \mathbb{P}\left( \sup_{a_N \leq u \leq 1-a_N}  {\{u(1-u)\}^{-\alpha}} \left|Q_N'(u,v_j) - \sqrt{v_j(1-v_j)}W_j(u)\right| > \varepsilon\right)                                                \\
     & \quad \leq \mathbb{P}\left( \sup_{a_N \leq u \leq 1-a_N}  \left|\sqrt{N}\left\{\tilde{F}_{N,v_j}(u) - u\right\} - W_j(u)\right| > \frac{\varepsilon \, a_N^\alpha(1-a_N)^\alpha}{\sqrt{v_j(1-v_j)}}\right) \\
     & \quad \leq \mathbb{P}\left( \sup_{a_N \leq u \leq 1-a_N}  \left|\sqrt{N}\left\{\tilde{F}_{N,v_j}(u) - u\right\} - W_j(u)\right| > \, \varepsilon \,  a_N^\alpha\right)              \\
     & \quad \leq c_2 \exp\left(-c_3 \sqrt{N} \Delta_N\right)
  \end{align*}
where we used that $(1-a_N)^\alpha/\sqrt{v_j(1-v_j)} \ge 1$ for $N$ sufficiently large and the last inequality follows by the tail bound \eqref{eq:hungarian_embedding_tail_bound} on the Hungarian embedding, where $\Delta_N$ is defined in \eqref{hd10} and satisfies $\sqrt{N}\Delta_N/\log N \to \infty$ since $\alpha < 1/4$. 
  This is essentially the same argument as in the proof of Eq.\@ \eqref{eq:QN_QNstrich_abstand}. Hence,
  \begin{align}
  \begin{split}
      \label{eq:probs_convergence_2}
     &N^2 \mathbb{P}\left(\sup_{a_N \leq u \leq 1-a_N}  {\{u(1-u)\}^{-\alpha}} \left|Q_N'(u,v_j) - \sqrt{v_j(1-v_j)}W_j(u)\right| > \varepsilon\right) \\
     &\quad \leq c_2 N^2 \exp\left(-c_3 \sqrt{N} \Delta_N\right) \xrightarrow[N \to \infty]{} 0.
  \end{split}
  \end{align}
Eqs.\@ \eqref{eq:probs_convergence_1} and \eqref{eq:probs_convergence_2} imply that the right-hand side of Eq.\@ \eqref{eq:diskretisierung_punktweise} tends to $0$ for $N \to \infty$. This proves Eq.\@ \eqref{eq:QN_prime_u_discrete-help}. 

\bigskip

\textit{Proof of~\eqref{eq:truncneg1} and~\eqref{eq:truncneg2}}
The bounds~\eqref{eq:truncneg1}, ~\eqref{eq:truncneg2} can be proved by similar arguments, so we will only provide a proof for~\eqref{eq:truncneg1}. Let $\{U_j\}_{j =1,\dots, N(N^2-1)}$ denote the collection of possibly dependent standard uniform random variables that make up $Q_N'(\cdot,v_1),\dots,Q_N'(\cdot,v_{N^2-1})$. Then by a union bound argument
\[
\P\big(\min_j U_j < N^{-4}\big) \le N^3N^{-4} = o(1).
\]
This implies that, with probability going to one, $Q_n'(u,v_j) = -u\sqrt{Nv_j(1-v_j)}$ for all $u \in [0,N^{-4}]$ and $j \in 1,\dots,N^2-1$. This in turn yields
\[
\max_{j=1,\dots,N^2-1} \sup_{0 \le u \le N^{-4}} |u^{-\alpha} Q_N'(u,v_j)| = o_\P(1),
\]
and it remains to consider the maximum over $N^{-4} \le u \le a_N$. To this end, we will apply part (6) of Inequality 1 in Chapter 11, Section 2 of \cite{shorack_wellner:1986} with, in their notation $a = N^{-4}, b= a_N, \delta = a_N$, $q(u) = u^\alpha$ and $\lambda = a_N^\alpha$. The function $\psi$ appearing in that inequality can be bounded by equation (10) in Proposition 1 in Chapter 11, Section 1  of \cite{shorack_wellner:1986} in the following way: $\psi(\lambda) \ge (1+\lambda/3)^{-1}$. Thus, writing
\[
\gamma^+(t) = \psi\left(\frac{a_N^\alpha}{\sqrt{N}t^{1-\alpha}}\right),
\]
which is the function from (8) in Chapter 11, Section 2 of \cite{shorack_wellner:1986}, it holds for all $t \in [N^{-4},a_N]$,
\begin{align*}
\lambda^2\gamma^+(t)q^2(t)/t &= a_N^{2\alpha}t^{2\alpha-1}\psi\left(\frac{a_N^{\alpha}}{\sqrt{N}t^{1-\alpha}}\right) \ge \frac{a_N^{2\alpha}t^{2\alpha-1}}{1 + a_N^\alpha N^{-1/2}t^{\alpha-1}/3} \ge \frac{t^{4\alpha-1} }{1 + N^{-1/2}t^{2\alpha-1}}
\\
&\ge \frac{1}{t^{1-4\alpha} + N^{-1/2}t^{-2\alpha} } \ge \frac{1}{2(a_N^{1/2}\vee N^{-1/4})}.
\end{align*}
In the last line we used $\alpha \le 1/32$ which implies $t^{-2\alpha} \leq t^{-1/16} \le N^{1/4}$ for $N^{-4} \le t$, so that $N^{-1/2}t^{-2\alpha} \le N^{-1/4}$ and further $t^{1-4\alpha} \le t^{1/2} \le a_N^{1/2}$ for $t \le a_N$. Now assume w.o.l.g. that $N$ is sufficiently large so that $a_N < 1/2$. Then by part (6) of Inequality 1 in Chapter 11, Section 2 of \cite{shorack_wellner:1986}, with the choices described above, 
\begin{align*}
\max_{j = 1,\dots, N^2-1}\P\Big(\sup_{u \in [N^{-4},a_N]} u^{-\alpha} Q_n'(u,v_j) \ge a_N^{\alpha}\Big) 
&\le 4\int_{N^{-4}/2}^{a_N} \frac{1}{t} 
\exp(-1/[128\sqrt{a_N}]) dt 
\\
&\le 4 \exp(-1/[128\sqrt{a_N}]) \log(N^{-4}/2) = o(N^{-2}). 
\end{align*}
Similar but simpler computations show that, applying part (6) of Inequality 1 in Chapter 11, Section 2 of \cite{shorack_wellner:1986} again for the negative part
\[
\max_{j = 1,\dots, N^2-1}\P\Big(\sup_{u \in [N^{-4},a_N]} \big\{-u^{-\alpha} Q_n'(u,v_j)\big\} \ge a_N^{\alpha}\Big) = o(N^{-2}).
\]
Combining the results above completes the proof of~\eqref{eq:truncneg1}. The proof for~\eqref{eq:truncneg2} can be done by similar arguments, and this completes the proof. 
\end{proof}

\begin{lemma}
  \label{lem:konvergenz_QN_prime}
Assume $N a_N \to \infty$ for $N \to \infty$. Additionally assume that $a_N = O(N^{-\eta})$ for some $\eta >0$. Fix $\alpha \in (0,1/32]$. Defining the process $L$ by
  $$
  L(u,v) = \frac{\sqrt{v(1-v)}}{\{u(1-u)\}^\alpha} G(u,v),
  $$
  where $G$ is the Gaussian process from the statement of Theorem \ref{thm:asymptotik}, it holds that
  \begin{equation}
    \label{eq:step3_3}
    \left[\{u(1-u)\}^{-\alpha} Q_N'(u,g_N(v))\right]_{0 \leq u,v \leq 1} \rightsquigarrow L.
  \end{equation}
\end{lemma}
\begin{proof}[Proof of Lemma~\ref{lem:konvergenz_QN_prime}] ~~~

In light of Lemma~\ref{lem:discretQN'inu} it suffices to prove process convergence of the completely discretized process $\left[\{g_N(u)(1-g_N(u))\}^{-\alpha} Q_N'(g_N(u),g_N(v))\right]_{0 \leq u,v \leq 1}$. As an intermediate step, we prove convergence of $\{g_N(u)(1-g_N(u))\}^{-\alpha} Q_N'(g_N(u),g_N(v))$ not as a process on the full unit square $[0,1]^2$ but only on $[0,1] \times [\varepsilon, 1-\varepsilon]$ for any fixed $0 < \varepsilon < 1/2$: We claim that
  \begin{equation}
    \label{eq:epsilon_prozesskonvergenz}
    \left[\{g_N(u)(1-g_N(u))\}^{-\alpha}Q_N'(g_N(u),g_N(v))\right]_{0 \leq u \leq 1, \varepsilon \leq v \leq 1-\varepsilon} \rightsquigarrow (L(u,v))_{0 \leq u \leq 1, \varepsilon \leq v \leq 1-\varepsilon}.
  \end{equation}
  
We will begin by proving asymptotic equicontinuity in probability. Let $T_N = \{j N^{-2} ~|~ j = 0, \ldots, N^2\}$ and $w,w'$ be shorthand for $(u,v)$ and $(u', v')$, respectively. Then
  \begin{align}
    \begin{split}
      \label{eq:supbound_gewichtet}
      & \sup_{w,w'} \left|\{g_N(u)(1-g_N(u))\}^{-\alpha}Q_N'(g_N(u),g_N(v)) - \{g_N(u')(1-g_N(u'))\}^{-\alpha}Q_N'(g_N(u'),g_N(v'))\right|                \\
      & \quad \leq \sup_{w,w' \in ([0,1]\times [\varepsilon, 1-\varepsilon])\cap T_N^2 : \|w-w'\|_\infty \leq \delta + 2 N^{-2}} \left|\{u(1-u)\}^{-\alpha}Q_N'(u,v) - \{u'(1-u')\}^{-\alpha}Q_N'(u',v')\right|,
    \end{split}
  \end{align}
  where the supremum in the first line is taken over all $w,w' \in [0,1]\times [\varepsilon, 1-\varepsilon]$ such that $\|w-w'\|_\infty \leq \delta$. The inequality holds because the function $g_N$ moves any point by at most $1/N^2$. Using the identity $xy - ab = x(y-b) + b(x-a)$ and Lemma \ref{lem:moment_lemma_FN_tilde}, we have
  \begin{align}
    \begin{split}
      \label{eq:momentbound_gewichtet}
      & \left\|\{u(1-u)\}^{-\alpha}Q_N'(u,v) - \{u'(1-u')\}^{-\alpha}Q_N'(u',v')\right\|_{L_{2k}}                                                                   \\
      & \quad \leq \left\|\{u(1-u)\}^{-\alpha} \left(\tilde{F}_{N,v}(u) - u\right) - \{u'(1-u')\}^{-\alpha} \left(\tilde{F}_{N,v'}(u') - u'\right)\right\|_{L_{2k}} \\
      & \qquad + \left|\sqrt{v(1-v)} - \sqrt{v'(1-v')}\right| \left\|\{u'(1-u')\}^{-\alpha} \left(\tilde{F}_{N,v'}(u') - u'\right)\right\|_{L_{2k}}                 \\
      & \quad \leq c_k \max\left\{N^{-1/2 + \alpha}, |u-u'|^\gamma + \tilde c_\eps |v-v'|^\gamma\right\}
    \end{split}
  \end{align}
  for all $0 \leq u,u' \leq 1$ and $\varepsilon \leq v, v' \leq 1-\varepsilon$ for any $1/4 \ge \gamma > 0$ which implies $|u-u'|^\gamma \geq |u-u'|^{1/4} \geq|u-u'|^{1/2-\alpha}$ and $|v-v'|^\gamma \geq |v-v'|^{1/4}$. In the last inequality, we have used the fact that $\left\|\{u'(1-u')\}^{-\alpha} \left(\tilde{F}_{N,v'}(u') - u'\right)\right\|_{L_{2k}}$ is bounded uniformly in $u'$ and $v'$ by Lemma~\ref{lem:moment_lemma_FN_tilde} applied with $u = 0$, combined with Lipschitz continuity of $u \mapsto \sqrt{u(1-u)}$ on $[\eps,1-\eps]$.

  Because for any $w,w' \in ([0,1]\times [\varepsilon, 1-\varepsilon])\cap T_N^2$, $w \neq w'$ implies that $\|w-w'\|_\infty \geq N^{-2}$, which is the step size of the grid, we have for such $w, w'$ and all $\alpha \leq 1/4$ that $N^{-1/2 + \alpha} \leq N^{-1/4} = (N^{-2})^{1/8} \leq \|w-w\|_\infty^\gamma$, provided that we chose $\gamma < 1/8$. Eq.\@ \eqref{eq:momentbound_gewichtet} therefore reduces to
  \begin{equation}
    \label{eq:momentbound_gewichtet_2}
    \left\|\{u(1-u)\}^{-\alpha}Q_N'(u,v) - \{u'(1-u')\}^{-\alpha}Q_N'(u',v')\right\|_{L_{2k}} \leq {C_k} \,d(w,w'),
  \end{equation}
  for some constant ${C_k}$ uniform in $0 \leq u,u' \leq 1$ and $\varepsilon \leq v,v' \leq 1-\varepsilon$, where $d(w,w') = |u-u'|^\gamma + |v-v'|^\gamma$. The covering number $N(\varepsilon, d)$ of the semimetric $d$ {defined on the unit square} is bounded by $\mathrm{const}\cdot \varepsilon^{-2/\gamma}$. Hence $N(\varepsilon,d)^{1/(2k)} \leq \mathrm{const} \cdot \varepsilon^{-1/(\gamma k)}$ is integrable with respect to $\varepsilon$ for any $k > 1/\gamma$. Eq.\@ \eqref{eq:momentbound_gewichtet} allows us to apply Theorem 2.2.4 in \cite{vandervaart_wellner:weak_convergence}, which together with Eq.\@ \eqref{eq:supbound_gewichtet} yields
  \begin{align}
    \begin{split}
      \label{eq:moment_sup_bound_gewichtet}
      & \left\|\sup_{{(w,w')\in S(\eps,\delta)}} \left|\{g_N(u)(1-g_N(u))\}^{-\alpha}Q_N'(g_N(u),g_N(v)) - \{g_N(u')(1-g_N(u'))\}^{-\alpha}Q_N'(g_N(u'),g_N(v'))\right|\right\|_{L_{2k}} \\
      & \quad \leq  D\left\{\int_0^\eta \varepsilon^{-1/(\gamma k)} ~\mathrm{d}\varepsilon + {2(\delta+2N^{-2})^\gamma} \eta^{-2/(\gamma k)}\right\}
    \end{split}
  \end{align}
  for some constant $D > 0$ and any $\eta > 0$, where the supremum in the first line is taken over all
  \[
  (w,w') \in S(\eps,\delta) := \Big\{(w,w'): w,w' \in [0,1]\times [\varepsilon, 1-\varepsilon], \|w-w'\|_\infty \le \delta \Big\}
  \]
which implies $d(w,w') \le 2(\delta+2N^{-2})^\gamma$. 
  Let $\mathrm{LHS}$ denote the left-hand side of Eq.\@ \eqref{eq:supbound_gewichtet}. Using $\eta = \eta(N,\delta) = (2(\delta+2N^{-2})^\gamma)^{\gamma k/4}$ in Eq.\@ \eqref{eq:moment_sup_bound_gewichtet} then shows that
  $$
    \lim_{\delta \downarrow 0} \limsup_{n \to \infty} \mathbb{P}(\mathrm{LHS} > x) \leq \lim_{\delta \downarrow 0} \limsup_{n \to \infty} D x^{-2k} \|\mathrm{LHS}\|_{2k}^{2k} = 0,
  $$
  But this is exactly asymptotic equicontinuity in probability of the process on the left-hand side in Eq.\@ \eqref{eq:epsilon_prozesskonvergenz}. The convergence of the finite dimensional projections follows from a standard multivariate Lindeberg theorem, e.g.\@ Proposition 2.27 in \cite{vandervaart:asymptotic_statistics}. This proves Eq.\@ \eqref{eq:epsilon_prozesskonvergenz}.

\bigskip

  It still remains to expand the convergence in Eq.\@ \eqref{eq:epsilon_prozesskonvergenz} to the process indexed in $0 \leq u,v \leq 1$. To this end, we claim that
  \begin{equation}
    \label{eq:epsilon_remainder_1}
    \lim_{\varepsilon \downarrow 0} \limsup_{n \to \infty} \mathbb{P}\left(\sup_{0 \leq v \leq \varepsilon} \sup_{0 \leq u \leq 1} \left|\{g_N(u)(1-g_N(u))\}^{-\alpha}Q_N'(g_N(u),g_N(v))\right| > \eta \right) = 0
  \end{equation}
  and
  \begin{equation}
    \label{eq:epsilon_remainder_2}
    \lim_{\varepsilon \downarrow 0} \limsup_{n \to \infty} \mathbb{P}\left(\sup_{1-\varepsilon \leq v \leq 1} \sup_{0 \leq u \leq 1} \left|\{g_N(u)(1-g_N(u))\}^{-\alpha}Q_N'(g_N(u),g_N(v))\right| > \eta \right) = 0
  \end{equation}
  for any $\eta > 0$. We consider Eq.\@ \eqref{eq:epsilon_remainder_1}, as the claim in Eq.\@ \eqref{eq:epsilon_remainder_2} can be proven analogously. By noting that $\{g_N(u)(1-g_N(u))\}^{-\alpha}Q_N'(g_N(u),0) \equiv 0$ and partitioning the interval $(0,\varepsilon]$ we get
  \begin{align}
    \begin{split}
      \label{eq:epsilon_remainder_3}
      & \sup_{0 \leq v \leq \varepsilon} \sup_{0 \leq u \leq 1} \left|\{g_N(u)(1-g_N(u))\}^{-\alpha}Q_N'(g_N(u),g_N(v))\right|             \\
      & \quad \leq \max_{k = 1, \ldots, K(\varepsilon,N)} \sup_{0 \leq u \leq 1} \left|\{g_N(u)(1-g_N(u))\}^{-\alpha}Q_N'\left(g_N(u),g_N\left(2^{-k}\varepsilon\right)\right)\right| \\
      & \qquad + \max_{k = 1, \ldots, K(\varepsilon,N)} \max_{  2^{-k} \varepsilon \leq v \leq 2^{-k+1} \varepsilon
      } \sup_{0 \leq u \leq 1} \left|\{g_N(u)(1-g_N(u))\}^{-\alpha}\Delta_N(u,v)\right| \\
      &\quad =: \mathrm{I}_N + \mathrm{II}_N,
    \end{split}
  \end{align}
  where $K(\varepsilon,N) = \lceil2 \log_2(N) + \log_2(\varepsilon)\rceil$, and
  $$
  \Delta_N(u,v) := Q_N'\left(g_N(u),g_N\left(v\right)\right) - Q_N'\left(g_N(u),g_N\left(2^{-k}\varepsilon\right)\right).
  $$
  The fact that the maxima over $k$ in Eq.\@ \eqref{eq:epsilon_remainder_3} do not need to run over all $k \in \mathbb{N}$ to make this inequality hold is due to the observation that $g_N\left(2^{-k}\varepsilon\right)$ remains unchanged once $k \geq K(\varepsilon,N)$, since then $2^{-k}\varepsilon \leq N^{-2}$. Let us first consider the term $\mathrm{I}_N$. For any fixed $0 \leq v \leq 1$,
  \begin{align*}
     & \left\|\sup_{0 \leq u \leq 1} \left|\{g_N(u)(1-g_N(u))\}^{-\alpha}Q_N'\left(g_N(u), v\right)\right|\right\|_{L_{2m}}                                   \\
     & \quad = 2 \sqrt{v(1-v)} \left\|\sup_{0 \leq u \leq 1} \left|\{g_N(u)(1-g_N(u))\}^{-\alpha}Q_N'\left(g_N(u),\frac{1}{2}\right)\right|\right\|_{L_{2m}},
  \end{align*}
  since for fixed $v$, the only influence on the distribution and thus the $L_{2m}$ norm is through the factor $\sqrt{v(1-v)}$ in the definition of $Q_N'$. By Eq.\@ \eqref{eq:momentbound_gewichtet_2}, 
  \begin{align*}
     & \left\|\{g_N(u)(1-g_N(u))\}^{-\alpha}Q_N'\left(g_N(u),\frac{1}{2}\right) - \{g_N(u')(1-g_N(u'))\}^{-\alpha}Q_N'\left(g_N(u'),\frac{1}{2}\right)\right\|_{L_{2m}} \\
     & \quad \leq {C_m} |g_N(u)-g_N(u')|^\gamma
  \end{align*}
  for any $\gamma < 1/8$ and a constant ${C_m}$ uniform in $0 \leq u \leq 1$. By using Eq.\@ \eqref{eq:moment_sup_bound_gewichtet} with $\delta = 2$ and $\varepsilon = 1/4$, we see that
  $$
    \left\|\sup_{0 \leq u \leq 1} \left|\{g_N(u)(1-g_N(u))\}^{-\alpha}Q_N'\left(g_N(u),\frac{1}{2}\right)\right|\right\|_{L_{2m}} \leq K_m
  $$
  for some constant ${K_m}$ uniform in $N$. Finally, it is easy to see that
  $$
    \sqrt{g_N\left(2^{-k}\varepsilon\right)\left\{1-g_N\left(2^{-k}\varepsilon\right)\right\}} \leq \sqrt{2^{-k} \varepsilon + N^{-2}}
  $$
  Therefore, for some sufficiently large $m$,
  \begin{align*}
    \|\mathrm{I}_N\|_{L_{2m}} & \leq \sum_{k=1}^{K(\varepsilon,N)} \left\|\sup_{0 \leq u \leq 1} \left|\{g_N(u)(1-g_N(u))\}^{-\alpha}Q_N'\left(g_N(u),g_N\left(2^{-k}\varepsilon\right)\right)\right|\right\|_{L_{2m}}     \\
    & \leq \sum_{k=1}^{K(\varepsilon,N)}  2 \sqrt{2^{-k}\varepsilon + N^{-2}} \left\|\sup_{0 \leq u \leq 1} \left|\{g_N(u)(1-g_N(u))\}^{-\alpha}Q_N'\left(g_N(u),\frac{1}{2}\right)\right|\right\|_{L_{2m}} \\
                              & \lesssim \sqrt{\varepsilon} \sum_{k=1}^\infty \sqrt{2^{-k}} + N^{-1} \lceil 2\log_2 N + \log_2 \varepsilon\rceil,
  \end{align*}
  where $\lesssim$ is hiding constants depending only on $m$, and so
  \begin{equation}
    \label{eq:IN_zero_in_prob}
    \lim_{\varepsilon \downarrow 0} \limsup_{n \to \infty} \mathbb{P}(\mathrm{I}_N > \eta/2) = 0
  \end{equation}
  for any $\eta > 0$ by Markov's inequality.

  Let us now consider $\mathrm{II}_N$. We can bound $\Delta_N(u,v)$ by observing that for $v \in [2^{-k}\eps,2^{-k+1}\eps]$ 
  \begin{align}
    \begin{split}
      \label{eq:delta_bound}
      \Delta_N(u,v) & \leq \left|\sqrt{g_N(v)(1-g_N(v))} - \sqrt{g_N\left(2^{-k}\varepsilon\right)\left\{1-g_N\left(2^{-k}\varepsilon\right)\right\}}\right| \left|\sqrt{N}\left\{\tilde{F}_{N,g_N\left(2^{-k}\varepsilon\right)}(g_N(u)) - g_N(u)\right\}\right|             \\
      & \quad + \sqrt{{g_N(v)(1-g_N(v))}
      } \left|\sqrt{N}\left\{\tilde{F}_{N,v}(g_N(u)) - g_N(u)\right\} - \sqrt{N}\left\{\tilde{F}_{N,g_N\left(2^{-k}\varepsilon\right)}(g_N(u)) - g_N(u)\right\}\right| \\
      & \lesssim \sqrt{\left(2^{-k} \varepsilon + N^{-2}\right)} \left|\sqrt{N}\left\{\tilde{F}_{N,g_N\left(2^{-k}\varepsilon\right)}(g_N(u)) - g_N(u)\right\}\right|                          \\
      & \qquad + \sqrt{{2^{-k}} \varepsilon + N^{-2}} \left|\sqrt{N}\left\{\tilde{F}_{N,v}(g_N(u)) - g_N(u)\right\} - \sqrt{N}\left\{\tilde{F}_{N,g_N\left(2^{-k}\varepsilon\right)}(g_N(u)) - g_N(u)\right\}\right| \\
      & \lesssim \sqrt{2^{-k} \varepsilon} \left|\sqrt{N}\left\{\tilde{F}_{N,g_N\left(2^{-k}\varepsilon\right)}(g_N(u)) - g_N(u)\right\}\right|                           \\
      & \qquad + \sqrt{2^{-k} \varepsilon} \left|\sqrt{N}\left\{\tilde{F}_{N,v}(g_N(u)) - g_N(u)\right\} - \sqrt{N}\left\{\tilde{F}_{N,g_N\left(2^{-k}\varepsilon\right)}(g_N(u)) - g_N(u)\right\}\right|,
    \end{split}
  \end{align}
where $\lesssim$ is hiding universal constants and we used the crude bound
\[
\Big| \sqrt{g_N(v)(1-g_N(v))} - \sqrt{g_N\left(2^{-k}\varepsilon\right)\left\{1-g_N\left(2^{-k}\varepsilon\right)\right\}} \Big| \lesssim \sqrt{2^{-k} \eps + N^{-2}}
\]
for the first term. We obtained the third inequality in Eq.\@ \eqref{eq:delta_bound} by observing $N^{-2} \lesssim 2^{-k}\varepsilon$ for all $k \leq K(\varepsilon, N)$. Writing
  $$
    T_k(u,v) = \left\{g_N(u)(1-g_N(u))\right\}^{-\alpha} \sqrt{N} \left[\left\{\tilde{F}_{N,v}(g_N(u)) - g_N(u)\right\} - \left\{\tilde{F}_{N,g_N\left(2^{-k}\varepsilon\right)}(g_N(u)) - g_N(u)\right\}\right],
  $$
  it therefore suffices to consider the processes $(u,v) \mapsto \sqrt{2^{-k} \varepsilon}\left|T_k(u,v)\right|$, since the first summand on the right-hand side in Eq.\@ \eqref{eq:delta_bound} can be treated the same way as $I_N$.

  Consider some $N^{-2} < \delta < 1/10$. For any $0 \leq u \leq 1$ and $\delta < v \leq 2\delta$, it holds that $\delta \leq g_N(\delta) \leq 2\delta$, $v \leq g_N(v) \leq 2\delta + N^{-2} < 3\delta$ and $|g_N(v) - g_N(\delta)| \leq 2\delta$. Furthermore, for $0 < v,v' < 1$, it can be easily proven that the variance $\sigma_{v,v'}^2$ {from Eq.\@ \eqref{eq:identities_variance_covariance_correlation}} is strictly smaller than $2$. Thus, by Eq.\@ \eqref{eq:l2k_norm_nichts_gleich}, there is a constant $C$ and some sufficiently small $\gamma > 0$ such that for $0 \leq u, u' \leq 1$ and $\delta \leq v,v' \leq 2\delta$,
  \begin{align}
    \begin{split}
      \label{eq:l2m_bound_deltaskaliert_1}
      & \delta^{1/4} \sqrt{N}\left\|\frac{\tilde{F}_{N,g_N(v)}(g_N(u)) - g_N(u)}{\{g_N(u)(1 - g_N(u))\}^{\alpha}} -  \frac{\tilde{F}_{N,g_N(v')}(g_N(u')) - g_N(u')}{\{g_N(u')(1 - g_N(u'))\}^{\alpha}}\right\|_{L_{2m}} \\
      & \quad \leq \delta^{1/4} C\,\max\left\{N^{-1/4}, |g_N(u) - g_N(u')|^\gamma + \sigma_{g_N(v),g_N(v')}^{1/2}\right\}          \\
      & \quad \leq C\left[\max\left\{N^{-1/4}, |g_N(u) - g_N(u')|^\gamma \right\} + \max\left\{N^{-1/4}, \delta^{1/4}\sigma_{g_N(v),g_N(v')}^{1/2}\right\}\right].
    \end{split}
  \end{align}
By Eq.\@ \eqref{eq:sigma2_formula_2}, we have 
$$
\sigma_{g_N(v), g_N(v')}^2 \leq |g_N(v) - g_N(v')| \frac{4}{g_N(v') [1 - g_N(v)]}.
$$
Since $\delta \leq v, v' \leq 2 \delta$ and $\delta < 1/10$, we have $g_N(v)[1 - g_N(v')] \geq \delta [1 - 1/10 - N^{-2}] \geq 7\delta/20$ for all $N \geq 2$. This also implies
$$
\delta^{1/4}\sigma_{g_N(v), g_N(v')}^{1/2} \leq \tilde{C} |g_N(v) - g_N(v')|^{1/4}
$$
with $\tilde{C} = (80/7)^{1/4}$. Eq.\@ \eqref{eq:l2m_bound_deltaskaliert_1} therefore reduces to
  \begin{align*}
     & \delta^{1/4} \sqrt{N}\left\|\frac{\tilde{F}_{N,g_N(v)}(g_N(u)) - g_N(u)}{\{g_N(u)(1 - g_N(u))\}^{\alpha}} -  \frac{\tilde{F}_{N,g_N(v')}(g_N(u')) - g_N(u')}{\{g_N(u')(1 - g_N(u'))\}^{\alpha}}\right\|_{L_{2m}} \\        
     & \quad \leq C\left[\max\left\{N^{-1/4}, |g_N(u) - g_N(u')|^\gamma \right\} + \max\left\{N^{-1/4}, |g_N(v) - g_N(v')|^{1/4}\right\}\right]                                                                         \\
     & \quad \leq C\left[\max\left\{N^{-1/4}, |g_N(u) - g_N(u')|^\gamma \right\} + \max\left\{N^{-1/4}, |g_N(v) - g_N(v')|^\gamma\right\}\right]
  \end{align*}
  if we choose $\gamma < 1/4$. The constant $C$ is not the same as in Eq.\@ \eqref{eq:l2m_bound_deltaskaliert_1} but still uniform in $N, u, u', v, v'$. If $w := (u,v) \neq (u', v') =: w'$, then at least one of the distances $|g_n(u) - g_N(u')$ or $|g_N(v) - g_N(v')|$ must be greater than or equal to $N^{-2}$. On the other hand, if $w = w'$, then the left-hand side of the above set of inequalities is $0$. Therefore,
  \begin{align}
    \begin{split}
      \label{eq:l2m_bound_deltaskaliert_2}
      & \delta^{1/4} \sqrt{N}\left\|\frac{\tilde{F}_{N,g_N(v)}(g_N(u)) - g_N(u)}{\{g_N(u)(1 - g_N(u))\}^{\alpha}} -  \frac{\tilde{F}_{N,g_N(v')}(g_N(u')) - g_N(u')}{\{g_N(u')(1 - g_N(u'))\}^{\alpha}}\right\|_{L_{2m}} \\
      & \quad \leq C\left\{|g_N(u) - g_N(u')|^\gamma + |g_N(v) - g_N(v')|^\gamma\right\}
    \end{split}
  \end{align}
  for some sufficiently small $\gamma > 0$ (this is essentially the same argument as in Eq.\@ \ref{eq:momentbound_gewichtet_2}). By the same arguments that led to Eq.\@ \eqref{eq:moment_sup_bound_gewichtet}  we obtain
  $$
    \delta^{1/4} \sqrt{N}\left\|\sup_{w,w' \in [0,1] \times [\delta, 2\delta]}\frac{\tilde{F}_{N,g_N(v)}(g_N(u)) - g_N(u)}{\{g_N(u)(1 - g_N(u))\}^{\alpha}} -  \frac{\tilde{F}_{N,g_N(v')}(g_N(u')) - g_N(u')}{\{g_N(u')(1 - g_N(u'))\}^{\alpha}}\right\|_{L_{2m}} \leq R
  $$
  for some sufficiently large $m \in \mathbb{N}$ and a constant $R$ independent of $N$ and $\delta$. Hence,
  \begin{align*}
     & \left\|\max_{k = 1, \ldots, K(\varepsilon, N)} \sup_{0 \leq u \leq 1} \sup_{2^{-k}\varepsilon \leq v \leq 2^{-k+1}\varepsilon} \sqrt{2^{-k} \varepsilon} \,|T_k(u,v)|\right\|_{L_{2m}}                                                              \\
     & \quad \leq \left\|\max_{k = 1, \ldots, K(\varepsilon, N)} \left(2^{-k} \varepsilon\right)^{1/4} \sup_{0 \leq u \leq 1} \sup_{2^{-k}\varepsilon \leq v \leq 2^{-k+1}\varepsilon} \left(2^{-k} \varepsilon\right)^{1/4} \,|T_k(u,v)|\right\|_{L_{2m}} \\
     & \quad \leq  \sum_{k=1}^\infty \left(2^{-k} \varepsilon\right)^{1/4} \left\| \sup_{0 \leq u \leq 1} \sup_{2^{-k}\varepsilon \leq v \leq 2^{-k+1}\varepsilon} \left(2^{-k} \varepsilon\right)^{1/4} \,|T_k(u,v)|\right\|_{L_{2m}}                      \\
     & \quad \leq \varepsilon^{1/4}  R \sum_{k=1}^\infty 2^{-k/4},
  \end{align*}
  and this bound converges to $0$ for $\varepsilon \to 0$. By Eq.\@ \eqref{eq:delta_bound} and the methods we used for $\mathrm{I}_N$, we therefore have
  \begin{equation}
    \label{eq:IIN_zero_in_prob}
    \lim_{\varepsilon \downarrow 0} \limsup_{n \to \infty} \mathbb{P}(\mathrm{II}_N > \eta/2) = 0
  \end{equation}
  for any $\eta > 0$. Together with Eqs.\@ \eqref{eq:IN_zero_in_prob} and \eqref{eq:epsilon_remainder_3}, this proves Eq.\@ \eqref{eq:epsilon_remainder_1}, and Eq.\@ \eqref{eq:epsilon_remainder_2} can be proven analogously.

  Now let us write
  \begin{align*}
    Y_N               & = \left(\{g_N(u)(1-g_N(u))\}^{-\alpha}Q_N'(g_N(u),g_N(v))\right)_{(u,v) \in [0,1]^2},                                   \\
    X_{N,\varepsilon} & = Y_N \cdot \textbf{1}_{[0,1] \times [\varepsilon, 1-\varepsilon]},\\
    L_\varepsilon     & = L \cdot \textbf{1}_{[0,1] \times [\varepsilon, 1-\varepsilon]}.                                                      
  \end{align*}
  We have shown in Eq.\@ \eqref{eq:epsilon_prozesskonvergenz} that $X_{N,\varepsilon}$ converges in distribution to $L_\varepsilon$ for every fixed $\varepsilon > 0$. Furthermore, Eqs.\@ \eqref{eq:epsilon_remainder_1} and \eqref{eq:epsilon_remainder_2} combined can be restated as
  $$
    \lim_{\varepsilon \downarrow 0} \limsup_{N \to \infty} \mathbb{P}\left(\sup_{0 \leq u,v \leq 1} |X_{N,\varepsilon}(u,v) - Y_N(u,v)| > \eta\right) = 0
  $$
  for any $\eta > 0$. Theorem 5.1 in \cite{dehling_etal:2014} now implies that $Y_N \rightsquigarrow L$ as $N \to \infty$. In light of Lemma~\ref{lem:discretQN'inu} this also proves Eq.\@ \eqref{eq:step3_3}.
\end{proof}

\begin{lemma}
\label{lem:uniform_quantiles}
Let $(a_N)_{N \in \mathbb{N}}$, be a sequence of positive numbers which satisfies $a_N = O(N^{-\kappa})$ for some $\kappa > 0$, and suppose that there exist two numbers $0 < \alpha \le 1/32$ and $\varepsilon > 0$ such that $N a_N^{1 + 2\alpha} \geq N^\varepsilon$ for almost all $N$. Then 
  $$
    \sup_{a_N \leq t \leq 1 - a_N} \sup_{0 \leq v \leq 1} \left| \int_{a_N}^t \sqrt{N}\left(F_{N,{g_N(v)}}^{-1}(u) - \Phi_{{g_N(v)}\{1-{g_N(v)}\}}^{-1}(u)\right) - \frac{Q_N'(u, g_N(v))}{\varphi \circ \Phi^{-1}(u)} ~\mathrm{d}u\right| \xrightarrow[N \to \infty]{\mathbb{P}} 0.
  $$
\end{lemma}
\begin{proof}[Proof of Lemma~\ref{lem:uniform_quantiles}]
We first consider the case $v = 0$, the case $v=1$ can be dealt with in the same way. Then $g_N(v) = 0$. By definition of $Q_N$, $Q_N(u,0) \equiv 0$. Also, $F_{N,{0}}^{-1}(u) = 0$ for all $u \in (0,1)$ almost surely and $\Phi_0^{-1}(u) = 0~\forall u \in (0,1)$. Thus for $v =0,1$ the difference inside the absolute vales is identically zero and it suffices to consider the supremum over $v \in (0,1)$. Then by definition of $g_N$, we have $g_N(v) \in \{1/N^2,\dots,1-1/N^2\}$. This will be utilized throughout the remaining proof.
  
Let us consider the case $a_N \leq t \leq 1/2$, as $1/2 \leq t \leq 1-a_N$ can be dealt with analogously. Define the event
$$
\Omega_N = \left\{u/2 \leq \tilde{F}_{N,g_N(v)}^{-1}(u) \textrm{ for all } 0 < v < 1 \textrm{ and } a_N \leq u \leq 1/2\right\}.
$$
Our first goal is to show that $\mathbb{P}(\Omega_N) \to 1$ as $N \to \infty$. Suppose that there exists some $v \in (0,1)$ and $u^* \in [a_N, 1/2]$ with $\tilde{F}_{N,g_N(v)}^{-1}(u^*) < u^*/2$. This implies
$$
\sqrt{N}\left|\tilde{F}_{N,g_N(v)}^{-1}(u^*) - u^*\right| \geq u^*/2,
$$
and since $u^* \geq a_N = N^{-\gamma}$, it must then also be true that
$$
\sup_{u \in [a_N, 1/2]} \sqrt{N}\left|\frac{\tilde{F}_{N,g_N(v)}^{-1}(u) - u}{u^{1/2 - \alpha}}\right| \geq \sqrt{N}\left|\frac{\tilde{F}_{N,g_N(v)}^{-1}(u^*) - u^*}{(u^*)^{1/2 - \alpha}}\right| > \sqrt{N} (u^*)^{1/2 + \alpha}/2 \geq \frac{1}{2}\sqrt{N}a _N^{1/2 + \alpha}.
$$
Since $g_N(v)$ can take at most $N^2$ values $v_j$, $j = 1, \ldots, N^2-1$, it is helpful to first consider one fixed value $g_N(v) = v_j \in (0,1)$. Let us write $V_N(u) = \sqrt{N}\big(\tilde{F}_{N,v_j}^{-1}(u) - u\big)$. $V_N$ is equal in distribution to the uniform quantile process. Therefore, by the Hungarian construction of the uniform quantile process \citep[Chapter 12, Section 2, Theorem 1 in][]{shorack_wellner:1986} there exist versions $V_N^*$ of $V_N$, along with Brownian bridges $B_N$, $N \in \mathbb{N}$ and positive constants $N_0, c_1, c_2, c_3 > 0$, such that
\begin{equation}
    \label{eq:dkw_quantile}
\mathbb{P}\left(\sup_{0 \leq u \leq 1} \left|V_N^*(u) - B_N(u)\right| \geq (c_1 \log N + x) N^{-1/2}\right) \leq c_2 \exp\left(-c_3 x\right)
\end{equation}
for all $N \in \mathbb{N}$ and $0 \leq x$. By the triangle inequality,
\begin{align}
\begin{split}
    \label{eq:weighted_triangle}
&\mathbb{P}\left(\sup_{u \in [a_N, 1/2]} \sqrt{N}\left|\frac{\tilde{F}_{N,v_j}^{-1}(u) - u}{u^{1/2 - \alpha}}\right| \geq \frac{1}{2}\sqrt{N}a _N^{1/2 + \alpha}\right) \\
&\leq \mathbb{P}\left(\sup_{u \in [a_N, 1/2]} \left|\frac{V_N^*(u) - B_N(u)}{u^{1/2 - \alpha}}\right| \geq \frac{1}{4}\sqrt{N}a _N^{1/2 + \alpha}\right) + \mathbb{P}\left(\sup_{u \in [a_N, 1/2]} \left|\frac{B_N(u)}{u^{1/2 - \alpha}}\right| \geq \frac{1}{4}\sqrt{N}a _N^{1/2 + \alpha}\right).
\end{split}
\end{align}
Let us consider the first probability on the right-hand side. Since $a_N \leq u$ it holds that $|V_N^*(u) - B_N(u)|/u^{1/2 - \alpha} \leq |V_N^*(u) - B_N(u)|/a_N^{1/2 - \alpha}$, and so 
\begin{align*}
&\mathbb{P}\left(\sup_{u \in [a_N, 1/2]} \left|\frac{V_N^*(u) - B_N(u)}{u^{1/2 - \alpha}}\right| \geq \frac{1}{4}\sqrt{N}a _N^{1/2 + \alpha}\right) \\
&\leq  \mathbb{P}\left(\sup_{u \in [a_N, 1/2]} \left|V_N^*(u) - B_N(u)\right| \geq \frac{1}{4}\sqrt{N}a _N\right) \\
&=  \mathbb{P}\left(\sup_{u \in [a_N, 1/2]} \left|V_N^*(u) - B_N(u)\right| \geq (c_1 \log N + x_N) N^{-1/2}\right),
\end{align*}
where $x_N = N a_N/4 - c_1 \log N$. It follows from our assumptions (noting that $a_N^{1+2\alpha} \leq a_N$ for almost all $N$) on the sequence $a_N$ that
$$
x_N = N a_N / 4 - c_1 \log N \geq N^\varepsilon/4 - N^{\varepsilon/2} \geq N^{\varepsilon/2}
$$
for sufficiently large $N$. Hence, by Eq.\@ \eqref{eq:dkw_quantile},
\begin{equation}
\label{eq:weighted_dkw_1}
    \mathbb{P}\left(\sup_{u \in [a_N, 1/2]} \left|\frac{V_N^*(u) - B_N(u)}{u^{1/2 - \alpha}}\right| \geq \frac{1}{4}\sqrt{N}a _N^{1/2 + \alpha}\right) \leq c_2 \exp\left(-c_3 N^{\varepsilon/2}\right)
\end{equation}
for almost all $N$. Now let us consider the second probability on the right-hand side of Eq.\@ \eqref{eq:weighted_triangle}. Since the processes $B_N$, $N \in \mathbb{N}$, are all Brownian bridges, we can use standard tail bounds for the weighted Brownian bridge. First, observe that
\begin{align}
\begin{split}
    \label{eq:weighted_bb_general}
    \mathbb{P}\left(\sup_{u \in [a_N, 1/2]} \left|\frac{B_N(u)}{u^{1/2 - \alpha}}\right| \geq \frac{1}{4}\sqrt{N}a _N^{1/2 + \alpha}\right) 
    &\leq \mathbb{P}\left(\sup_{0 \leq u \leq 1/2} \left|\frac{B_N(u)}{\{u(1-u)\}^{1/2 - \alpha}}\right| \geq \frac{1}{4}\sqrt{N}a _N^{1/2 + \alpha}\right).
\end{split}
\end{align}
Now write $\lambda_N = \sqrt{N}a _N^{1/2 + \alpha}/4$ and use Inequality 2 in Chapter 11, Section 2, of \cite{shorack_wellner:1986} with, in their notation, $a = 0, b= 1/2, \delta = 1/2, q(u) = \{u(1-u)\}^{1/2-\alpha}$, to bound 
\begin{align}
\begin{split}
    \label{eq:weighted_bb_general_2}
    \mathbb{P}\left(\sup_{0 \leq u \leq 1/2} \left|\frac{B_N(u)}{\{u(1-u)\}^{1/2 - \alpha}}\right| \geq \frac{1}{4}\sqrt{N}a _N^{1/2 + \alpha}\right) &\leq 8 \int_0^{1/2} t^{-1} \exp\left(-\lambda_N^2 \frac{\{t(1-t)\}^{1 - 2\alpha}}{16 t}\right) ~\mathrm{d}t \\
    &\leq 8 \int_0^{1/2} t^{-1} \exp\left(-\frac{\lambda_N^2}{32} t^{-2\alpha}\right) ~\mathrm{d}t .
\end{split}
\end{align}
Some algebra shows that for any $L > 0$ and $0 < t < 1$, $t^{-1} \exp(-L t^{-2\alpha}) \leq \exp(-L)$ is equivalent to $L \geq (\log t )/(1 - t^{-2\alpha})$. The expression on the right-hand side is continuous on $(0,1]$ and converges to $0$ as $t \downarrow 0$. It therefore attains a maximum value $m_\alpha < \infty$ on the compact set $[0,1]$, and so $t^{-1} \exp(-Lt^{-2\alpha}) \leq \exp(-L)$ certainly holds for all $t \in [0,1/2]$ if $L \geq m_\alpha$. Since $\lambda_N^2 \to \infty$ due to our assumptions on the sequence $a_N$, this implies for all sufficiently large $N$
$$
\int_0^{1/2} t^{-1} \exp\left(-\frac{\lambda_N^2}{32} t^{-2\alpha}\right) ~\mathrm{d}t \leq \frac{1}{2} \exp\left(-\frac{\lambda_N^2}{32}\right).
$$
Further, by the definition $\lambda_N = \sqrt{N}a _N^{1/2 + \alpha}/4$ and our assumptions on $a_N$ we have $\lambda_N^2 \ge N^\eps$ for almost all $N$, which together with Eqs.\@ \eqref{eq:weighted_bb_general} and \eqref{eq:weighted_bb_general_2} implies
\begin{equation}
\label{eq:weighted_bb}
    \mathbb{P}\left(\sup_{u \in [a_N, 1/2]} \left|\frac{B_N(u)}{u^{1/2 - \alpha}}\right| \geq \frac{1}{4}\sqrt{N}a _N^{1/2 + \alpha}\right) \leq 8 \exp\left(- \frac{N^\varepsilon}{32}\right)
\end{equation}
for almost all $N$. Now from Eqs.\@ \eqref{eq:weighted_triangle}, \eqref{eq:weighted_dkw_1} and \eqref{eq:weighted_bb} it follows that
\begin{align*}
    \mathbb{P}\left(\Omega_N^C\right) &\leq \mathbb{P}\left(\exists j = 1, \ldots, N^2-1 : \sup_{u \in [a_N, 1/2]} \sqrt{N}\left|\frac{\tilde{F}_{N,v_j}^{-1}(u) - u}{u^{1/2 - \alpha}}\right| \geq \frac{1}{2}\sqrt{N}a _N^{1/2 + \alpha}\right) \\
    &\leq \sum_{j=1}^{N^2-1} \mathbb{P}\left(\sup_{u \in [a_N, 1/2]} \sqrt{N}\left|\frac{\tilde{F}_{N,v_j}^{-1}(u) - u}{u^{1/2 - \alpha}}\right| \geq \frac{1}{2}\sqrt{N}a _N^{1/2 + \alpha}\right) \\
    &\lesssim N^2 \left[c_2 \exp\left(-c_3 N^{\varepsilon/2}\right) + 8 \exp\left(-\frac{N^\varepsilon}{32}\right)\right],
\end{align*}
where $\lesssim$ is hiding a universal constant, and so $\mathbb{P}(\Omega_N) \to 1$ as $N \to \infty$. 
By Eq.\@ \eqref{eq:quantilreduktion_2},
  \begin{align*}
     & \sqrt{N}\left(F_{N,g_N(v)}^{-1}(u) - \Phi_{g_N(v)\{1-g_N(v)\}}^{-1}(u)\right)       \\
     & \quad = \frac{Q_N(u,g_N(v))}{\varphi\left(\Phi^{-1}(u)\right)} + \left(\frac{1}{\varphi\left(\Phi^{-1}(\xi_{N,u,v})\right)} - \frac{1}{\varphi\left(\Phi^{-1}(u)\right)}\right) Q_N(u,g_N(v)),
  \end{align*}
  for some $\xi_{N,u,v}$ between $u$ and $\tilde{F}_{N,g_N(v)}^{-1}(u)$. On $\Omega_N$ we have the inequality
  $$
    \frac{1}{\varphi \circ \Phi^{-1}(\xi_{N,u,v})} \leq \frac{1}{\varphi \circ \Phi^{-1}(u/2)} \leq \frac{K}{u} 
  $$
  with the constant $K := \frac{2}{\Phi^{-1}(3/4)} > 1$ for all $a_N \leq u \leq 1/2$, because $1/(\varphi \circ \Phi^{-1})$ is antitone on $(0, 1/2]$. The last inequality follows from Mills ratio. Thus, for any $1/2 > \varepsilon > a_N$, $0<\eta<1/4$
  \begin{align}
    \begin{split}
      \label{eq:int_bound_1}
      & \sup_{0 < v < 1}\left|\int_{a_N}^\varepsilon \left(\frac{1}{\varphi\left(\Phi^{-1}(\xi_{N,u,v})\right)} - \frac{1}{\varphi\left(\Phi^{-1}(u)\right)}\right) Q_N(u,g_N(v)) ~\mathrm{d}u \right|                       \\
      & \quad \leq \sup_{0 < v < 1}\int_{a_N}^\varepsilon \frac{2K\{u(1-u)\}^\eta}{\varphi \circ \Phi^{-1}(u/2)} \left|\{u(1-u)\}^{-\eta} Q_N(u,g_N(v))\right| ~\mathrm{d}u                                               \\
      & \quad \leq  \Big(\sup_{0 < v < 1}\sup_{a_N \leq u \leq 1/2} \left|\{u(1-u)\}^{-\eta} Q_N(u,g_N(v))\right| \Big) \cdot \int_0^\varepsilon \frac{2K\{u(1-u)\}^\eta}{\varphi \circ \Phi^{-1}(u/2)} ~\mathrm{d}u \\
      & \quad = A_N  \int_0^\varepsilon \frac{\{u(1-u)\}^\eta}{\varphi \circ \Phi^{-1}(u/2)} ~\mathrm{d}u
    \end{split}
  \end{align}
  where $A_N$ independent of $\eps$ satisfies $A_N = \mathcal{O}_\mathbb{P}(1)$ by Lemma~\ref{lem:approxQNQn'} and Lemma~\ref{lem:konvergenz_QN_prime}, and the integral tends to $0$ for $\varepsilon \to 0$. On the other hand, Eq.\@ \eqref{eq:3logN_bound} implies
  $$
    \sup_{0 < v < 1} \sup_{\varepsilon \leq u \leq 1/2} \left|\frac{1}{\varphi\left(\Phi^{-1}(\xi_{N,u,v})\right)} - \frac{1}{\varphi\left(\Phi^{-1}(u)\right)}\right| = o_\mathbb{P}(1)
  $$
  by the uniform continuity of $1/(\varphi \circ \Phi^{-1})$ on the compact set $[\varepsilon, 1/2]$. Therefore,
  \begin{align}
    \begin{split}
      \label{eq:int_bound_2}
      & \sup_{v \in (0,1)} \left|\int_\varepsilon^{1/2} \left(\frac{1}{\varphi\left(\Phi^{-1}(\xi_{N,u,v})\right)} - \frac{1}{\varphi\left(\Phi^{-1}(u)\right)}\right) Q_N(u,g_N(v)) ~\mathrm{d}u \right|                       \\
      & \quad \leq \mathcal{O}_\mathbb{P}(1) \sup_{v \in (0,1)} \left(\frac{1}{2} - \varepsilon\right)  \sup_{v \in (0,1)} \sup_{\varepsilon \leq u \leq 1/2} \left|\frac{1}{\varphi\left(\Phi^{-1}(\xi_{N,u,v})\right)} - \frac{1}{\varphi\left(\Phi^{-1}(u)\right)}\right| \\
      &\quad =: B_N(\eps) = o_\mathbb{P}(1).
    \end{split}
  \end{align}
  Eqs.\@ \eqref{eq:int_bound_1} and \eqref{eq:int_bound_2} combine to yield
  \begin{align*}
     & \sup_{v \in (0,1)} \sup_{a_N \leq t \leq 1/2}\left|\int_{a_N}^t \left(\frac{1}{\varphi\left(\Phi^{-1}(\xi_{N,u,v})\right)} - \frac{1}{\varphi\left(\Phi^{-1}(u)\right)}\right) Q_N(u,g_N(v)) ~\mathrm{d}u \right| \\
     & \leq A_N \int_0^\varepsilon \frac{2\{u(1-u)\}^\alpha}{\varphi \circ \Phi^{-1}(u/2)} + B_N(\eps) 
  \end{align*}
  for any $\varepsilon > 0$. Since the integral converges to $0$ for $\varepsilon \to 0$, and since $A_N$ is independent of $\eps$ and $B_N(\eps) = o_\mathbb{P}(1)$ for any fixed $\eps > 0$,  this implies
  $$
    \sup_{v \in (0,1)} \sup_{a_N \leq t \leq 1/2}\left|\int_{a_N}^t \left(\frac{1}{\varphi\left(\Phi^{-1}(\xi_{N,u,v})\right)} - \frac{1}{\varphi\left(\Phi^{-1}(u)\right)}\right) Q_N(u,g_N(v)) ~\mathrm{d}u \right| = o_\mathbb{P}(1).
  $$
  For $1/2 \leq t \leq 1 - a_N$, we can proceed in a similar manner.
\end{proof}

\end{appendix}
\end{document}